\documentclass[16pt]{amsart}
\usepackage{amssymb, latexsym, amsmath, amsfonts}

\newtheorem{thm}{Theorem}[section]
\newtheorem{cor}[thm]{Corollary}
\newtheorem{lem}[thm]{Lemma}
\newtheorem{prop}[thm]{Proposition}
\theoremstyle{definition}

\theoremstyle{remark}
\newtheorem{rem}[thm]{Remark}
\numberwithin{equation}{section}

\newcommand{\mbf}{\mathbf}
\newcommand{\ra}{\rightarrow}
\newcommand{\z}{\zeta}
\newcommand{\pa}{\partial}
\newcommand{\ov}{\overline}
\newcommand{\sm}{\setminus}
\newcommand{\ep}{\epsilon}
\newcommand{\no}{\noindent}
\newcommand{\Om}{\Omega}
\newcommand{\cal}{\mathcal}
\newcommand{\ti}{\tilde}
\newcommand{\la}{\lambda}

\newcommand{\La}{\Lambda}
\newcommand{\al}{\alpha}
\newcommand{\be}{\beta}
\newcommand{\de}{\delta}
\newcommand{\sig}{\sigma}
\newcommand{\nab}{\nabla}
\newcommand{\om}{\omega}

\begin{document}
\title{Remarks on the metric induced by the Robin function}
\thanks{The second author was supported in part by a grant from UGC under DSA-SAP, Phase IV}
\author{Diganta Borah and Kaushal Verma}
\address{Department of Mathematics,
Indian Institute of Science, Bangalore 560 012, India}
\email{diganta@math.iisc.ernet.in, kverma@math.iisc.ernet.in}

\begin{abstract}
Let $D$ be a smoothly bounded pseudoconvex domain in $\mbf C^n$, $n > 1$. Using $G(z, p)$, the Green function for $D$ with pole at $p \in D$ associated with
the standard sum-of-squares Laplacian, N. Levenberg and H. Yamaguchi had constructed a K\"{a}hler metric (the so-called $\La$--metric) using the Robin
function $\La(p)$ arising from $G(z, p)$. The purpose of this article is to study this metric by deriving its boundary asymptotics and using them to
calculate the holomorphic sectional curvature along normal directions. It is also shown that the $\La$-metric is comparable to the Kobayashi (and hence
to the Bergman and Carath\'{e}odory metrics) when $D$ is strongly pseudoconvex. The unit ball in $\mbf C^n$ is also characterized among all smoothly bounded
strongly convex domains on which the $\La$-metric has constant negative holomorphic sectional curvature. This may be regarded as a version of Lu-Qi
Keng's theorem for the Bergman metric.
\end{abstract}

\maketitle

\section{Introduction}

\no Let $D$ be a smoothly bounded domain in $\mbf C^n$, $n > 1$ with smooth defining function $\psi(z)$ so that $D = \{ \psi < 0\}$. For $p \in D$, let $G(z,
p)$ be the Green function for $D$ with pole at $p$ associated to the standard Laplacian
\[
\Delta = 4 \sum_{j = 1}^{n} \frac{\pa^2}{\pa z_j \pa \ov z_j}
\]
on $\mbf C^n \approx \mbf R^{2n}$. The notation for the Green function will be enhanced to $G_D(z, p)$ if the need to emphasize its dependence on $D$ arises.
Then $G(z, p)$ is the unique function satisfying $\Delta G(z, p) = 0$ on $D \sm \{p\}$, $G(z, p) \ra 0$ as $z
\ra \pa D$
and $G(z, p) - \vert z - p \vert^{-2n + 2}$ is harmonic near $p$. Let $H_p(z)$ be harmonic on $D$, continuous on $\ov D$ and such that
$H_p(z) = - \vert z - p \vert^{-2n + 2}$ for $z \in \pa D$. The existence of such a function is guaranteed by the solution to the Dirichlet problem and
\[
G(z, p) = \vert z - p \vert^{-2n + 2} + H_p(z)
\]
is then the Green function for $D$ with pole at $p$. Therefore
\[
\La(p) := \lim_{z \ra p} \Big( G(z, p) - \vert z - p \vert^{-2n + 2} \Big)
\]
exists for each $p \in D$ and is called the Robin constant at $p$ and the map $p \mapsto \La(p)$ is the Robin function for $D$. Said differently $\La(p)
= H_p(0)$. Hence it is possible to rewrite the expression for $G(z, p)$ as
\begin{equation}
G(z, p) = \vert z - p \vert^{-2n + 2} + \La(p) + h_p(z)
\end{equation}
where $h_p(z) = H_p(z) - H_p(0)$ is harmonic on $D$ and satisfies $h_p(p) = 0$ for all $p \in D$. The maximum principle implies that $\La(p) < 0$ for all $p
\in D$ and it was shown by Yamaguchi in \cite{Y} that $p \mapsto - \La(p)$ is a real analytic exhaustion function for $D$. The assertion that $\La(p)$ is
real analytic is entirely local and as such does not require any smoothness hypothesis on $\pa D$ while the latter statement that $-\La(p)$ is an exhaustion
only needs the domain to be one where the Dirichlet problem can be solved. These claims are therefore valid on a much larger class of bounded domains than the
smoothly bounded ones, the point here of course being that the smoothness assumption on $\pa D$ does not play any vital role so far. Further understanding of
the Robin function came through the work of Levenberg-Yamaguchi (\cite{LY}) where they derived explicit expressions for the complex Hessians of $-\La$
and $\log(-\La)$ on smoothly bounded domains using the technique of variation of domains. To describe this briefly, fix $p_0 \in D$ and $a \in \mbf C^n$
a non-zero tangent vector at $p_0$. Let $\Delta_{\rho} \subset \mbf C$ be a disc around the origin with radius $\rho > 0$ such that $p_0 + a t \in D$ for
all $t \in \Delta_{\rho}$. Consider the biholomorphism $T : \Delta_{\rho} \times D \ra \cal D = T(\Delta_{\rho} \times D)$ defined by
\[
T(t, z) = (t, z - a t).
\]
$\cal D$ is a locally trivial family of domains $D_t = T(\{t\} \times D)$ such that $p_0 \in D_t$ for all $t \in \Delta_{\rho}$. The question of
understanding the complex Hessian of say $-\La$ at $p_0$ along $a$, which tantamounts to varying the pole $p$ near $p_0$, is transformed by $T$ to a
situation where the domains vary (i.e., $t \mapsto D_t$) but the pole $p_0 \in D_t$ remains fixed. Let $\la(t)$ be the Robin constant for $D_t$ at $p_0$.
It is not difficult to see that $\la(t) = \La(p_0 + a t)$ and hence
\[
\sum_{\al, \be = 1}^n \frac{\pa^2 (- \La)}{\pa z_{\al} \pa \ov z_{\be}} (p_0) a_{\al} \ov a_{\be} = \frac{\pa^2 (- \la)}{\pa t \pa \ov t} (0).
\]
The problem now is to understand the variation of $\la(t)$ as a function of $t$ and this is addressed in \cite{LY}. A similar idea can be used to
describe the complex Hessian of $\log(-\La)$. To get a flavor of what's involved, the end results may be summarized as follows: let $p_0$ and $a$ be as
above. Then
\begin{multline}
\sum_{\al, \be = 1}^{n} \frac{\pa^2 (-\La)}{\pa z_{\al} \ov z_{\be}} (p_0) a_{\al} \ov a_{\be} = \frac{1}{(n - 1) \sig_{2n}} \int_{\pa D} K_2(z, a) \big
\vert \nab_z G(z, p_0) \big \vert^2 \; dS_z \; + \\
\frac{4}{(n - 1) \sig_{2n}} \int_{D} \sum_{\al = 1}^n \Big \vert \sum_{\be = 1}^n a_{\be} \frac {\pa}{\pa \ov z_{\al}} G_{\be}(z, p_0) \Big \vert^2 \; dV_z
\end{multline}
where $\sig_{2n}$ is the surface area of the unit sphere in $\mbf R^{2n}$ and $K_2(z, a)$ is a quantity whose non-negativity (or strict positivity) is
equivalent to $z \in \pa D$ being a point of pseudoconvexity (or strong pseudoconvexity respectively) -- $K_2(z, a)$ is in fact manufactured from the Levi
form of $\pa D$ and can be shown to be independent of the defining function $\psi$. Moreover $\nab_z G(z, p) = \big( \pa G / \pa z_1 (z, p), \pa G / \pa
z_2 (z, p), \ldots, \pa G / \pa z_n (z, p) \big)$ for $p \in D$ is well defined on $\ov D$ since the boundary is assumed to be smooth, and
\[
G_{\be}(z, p) = \left( \frac{\pa G}{\pa p_{\be}} + \frac{\pa G}{\pa z_{\be}} \right)(z, p)
\]
for $1 \le \be \le n$. The expression for the Hessian of $\log(-\La)$ is similar:
\begin{multline}
\sum_{\al, \be = 1}^n \frac{\pa^2 \log(-\La)}{\pa z_{\al} \pa \ov z_{\be}} (p_0) a_{\al} \ov a_{\be} = \frac{1}{(n - 1) \sig_{2n} \big( -\La(p_0) \big)} \int_{\pa D}
K_2(z, \cal O) \big \vert \nab_z G(z, p_0) \big \vert^2 \; dS_z \; + \\
\frac{4}{(n - 1) \sig_{2n} \big( -\La(p_0) \big)} \int_D \sum_{\be = 1}^n \Big \vert \frac{\pa}{\pa \ov z_{\be}} H(a, p_0, z) \Big \vert^2 \; dV_z
\end{multline}
where the vector $\cal O$ depends on $a, p_0$ and the first derivatives of $\La$ at $p_0$ and $H(a, p_0, z)$ depends on $G(z, p_0)$ and $G_{\be}(z, p_0)$
defined above. While the exact expressions for $K_2(z, .)$ and $H(a, p_0, z)$ may be found in \cite{LY} it will suffice for our purposes to observe that
the volume integral in both (1.2) and (1.3) is non-negative. Thus the complex Hessians of $-\La$ and $\log(-\La)$ are bounded from below by quantities that
depend on $K_2(z, .)$. Since $D$ is smoothly bounded, the point that is farthest from, say the origin must be strongly pseudoconvex and
hence $K_2(z, .) > 0$ on a non-empty open piece of $\pa D$. If in addition it is assumed that $D$ is pseudoconvex then $K_2(z, .) \ge 0$ everywhere on
$\pa D$. It follows that $-\La$ and $\log(-\La)$ are both real analytic strongly plurisubharmonic exhaustions on smoothly bounded pseudoconvex domains.
Therefore on such a domain $\om_R  = i \pa \ov \pa \log(-\La)$ is the fundamental form of a K\"{a}hler metric that is induced by the Robin function. This
is the so-called $\La$--metric and is given by
\[
ds^2_z = \sum_{\al, \be = 1}^n \frac{\pa^2 \log(-\La)}{\pa z_{\al} \pa \ov z_{\be}} (z)
dz^{\al} \otimes d \ov z^{\be}.
\]
As usual we let $g_{\al \ov \be}(z) = \pa^2 \log(-\La) / \pa z_{\al} \pa \ov z_{\be} (z)$ for $1 \le \al, \be \le n$.
An interesting alternative approach for proving the plurisubharmonicity of $-\La$ was developed by Berndtsson in \cite{Ber}.

\medskip

As an example take $D = B(0, r)$ the euclidean ball of radius $r > 0$ in $\mbf C^n$. For $p \in B(0, r) \sm \{0\}$
\[
G_D(z, p) = \vert z - p \vert^{-2n + 2} - \big( r / \vert p \vert \big)^{2n - 2} \vert z - p^{\ast} \vert^{-2n + 2}
\]
where $p^{\ast} = \big(r / \vert p \vert \big)^2 p$ is the point that is symmetric to $p$ with respect to $\pa B(0, r)$. When $p$ is the origin the Green function
is $\vert z \vert^{-2n + 2} - 1$. The Robin function is
\[
\La(p) = - \big( r / \vert p \vert \big)^{2n - 2} \vert p - p^{\ast} \vert^{-2n + 2} = - r^{2n - 2} \big( r^2 - \vert p \vert^2 \big)^{-2n + 2} < 0
\]
for all $p \in B(0, r)$ and this looks like the Bergman kernel for the ball of radius $r > 0$ restricted to the diagonal albeit with a different exponent.
Therefore the $\La$--metric is, up to a constant, the Bergman metric in this case.

\medskip

\no The other example to consider is a half space given by $H = \big\{ z \in \mbf C^n : 2 \Re \big( \sum_{\al = 1}^n a_{\al} z_{\al} \big) - c < 0 \big\}$
where $a_{\al} \in \mbf C$ and $c$ is a real constant. For $p \in H$,
\[
G_H(z, p) = \vert z - p \vert^{-2n + 2} - \vert z - p^{\ast} \vert^{-2n + 2}
\]
where $p^{\ast}$ is the point that is symmetric to $p$ with respect to the real hyperplane $\pa H$. Hence
\begin{equation}
\La_H(p) = - \vert p - p^{\ast} \vert^{-2n + 2} = -\big( 2 \; {\rm dist}(p, \pa H) \big)^{-2n + 2} = \frac{ - \sum_{\al = 1}^n \vert a_{\al} \vert^{2n -
2}}{\Big(2 \Re \big(\sum_{\al = 1}^n a_{\al} p_{\al} \big) - c \Big)^{2n  - 2}} < 0.
\end{equation}
When $H = \{z \in \mbf C^n : \Re z_n < 0\}$, the Robin function $\La_H(p) = -2^{-2n + 2} (p_n + \ov p_n)^{-2n + 2}$ which evidently implies that $\om_R$ is
positive semi-definite and therefore the corresponding metric is degenerate.

\medskip

\no For an arbitrary smoothly bounded pseudoconvex domain the $\La$--metric is defined by a global potential namely $\log(-\La)$ and it is only natural to
hope that getting a hold on finer properties of $\La$ would yield more information about the metric. To this end, boundary asymptotics of $\La$ and its
derivatives up to order $3$ were computed in \cite{LY} in terms of the defining function $\psi$. As a consequence it was shown that the $\La$--metric is
complete on strongly pseudoconvex or smoothly bounded convex domains. For multi-indices $A = (\al_1, \al_2, \ldots, \al_n), B = (\be_1, \be_2, \ldots,
\be_n) \in \mbf N^n$ let
\[
D^{A} = \frac{{\pa}^{\vert A \vert}}{\pa z_1^{\al_1} \pa z_2^{\al_2} \cdots \pa z_n^{\al_n}}\; {\rm and} \; D^{\ov B} =  \frac{{\pa}^{\vert B
\vert}}{\pa \ov z_1^{\be_1} \pa \ov z_2^{\be_2} \cdots \pa \ov z_n^{\be_n}}
\]
and let $D^{A \ov B} = D^{A} D^{\ov B}$. Recall that a domain is called regular if there is a subharmonic barrier at each of its boundary points
in which case the Perron subsolution to the Dirichlet problem extends continuously to the boundary.

\begin{thm}
Let $D \subset \mbf C^n$ be a bounded regular domain and $\La$ the associated Robin function. Let $\Gamma \subset \pa D$ be a $C^2$-smooth open piece and
fix $z_0 \in \Gamma$. Suppose that $\psi$ is a $C^2$-smooth local defining function for $\Gamma$ near $z_0$. Define the half space
\[
\cal H = \bigg \{ z \in \mbf C^n : 2 \Re \Big( \sum_{\al = 1}^n \frac{\pa \psi}{\pa z_{\al}}(z_0) z_{\al} \Big) - 1 < 0 \bigg \}
\]
and let $\La_{\cal H}$ denote the Robin function associated to $\cal H$. Then for all multi-indices $A, B \in \mbf N^n$ and $z \in D$
\[
(-1)^{\vert A \vert + \vert B \vert} \big( D^{A \ov B} \La(z) \big) \big( \psi(z) \big)^{2n - 2 + \vert A \vert + \vert B \vert} \ra D^{A \ov B}
\La_{\cal H}(0)
\]
as $z \ra z_0$.
\end{thm}

\no This provides information about the boundary asymptotics of all derivatives of $\La$ since $D^{\al \ov \be}
\La_{\cal H}(0)$ can be explicitly computed for all multi-indices $A, B$ using (1.4) and in particular yields the following result from \cite{LY}.

\begin{cor}
Let $D \subset \mbf C^n$ be a $C^{\infty}$-smoothly bounded pseudoconvex domain (in which case the $\La$--metric is well defined) and fix $z_0 \in \pa D$. Then for $z
\in D$ and $1 \le \al, \be \le n$,
\[
g_{\al \ov \be}(z) \big( \psi(z) \big)^2 \ra (2n - 2) \frac{\pa \psi}{\pa z_{\al}}(z_0) \; \frac{\pa \psi}{\pa \ov z_{\be}}(z_0)
\]
as $z \ra z_0$. Moreover for $1 \le \al, \be, \gamma \le n$
\[
\frac{\pa^3 \La}{\pa z_{\al} \pa \ov z_{\be} \pa z_{\gamma}}(z) \big( \psi(z) \big)^{2n + 1} \ra 2n(2n - 1)(2n - 2) \; \frac{\pa \psi}{\pa z_{\al}}(z_0)
\; \frac{\pa \psi}{\pa \ov z_{\be}}(z_0) \frac{\pa \psi}{\pa \ov z_{\gamma}}(z_0) \; \big \vert \nab \psi(z_0) \big \vert^{2n - 2}
\]
as $z \ra z_0$.
\end{cor}

\no Invariant metrics such as those of Bergman, Carath\'{e}odory and Kobayashi play an important role in understanding the geometry of a given domain and
their equivalence on smoothly bounded strongly pseudoconvex domains is well known. While it is not clear to us whether the $\La$--metric is invariant under
biholomorphisms, the question of how it compares with one of these canonical metrics seems natural to consider. Let $D \subset \mbf C^n$ be a $C^{\infty}$-smoothly bounded domain and fix $z_0 \in \pa D$. Let $\psi$ be a $C^{\infty}$-smooth defining function for $D$. At each point $p \in \pa D$ there is a canonical splitting
$\mbf C^n = H_p(\pa D) \oplus N_p(\pa D)$ along the complex tangential and normal directions at $p$. A vector $v \in \mbf C^n$ regarded as a tangent vector at $p \in \pa D$ can therefore be uniquely written as $v = v_H(p) + v_N(p)$ where $v_H(p) \in H_p(\pa D)$ and $v_N(p) \in N_p(\pa D)$. The smoothness of $\pa D$ implies that if $z \in D$ is close enough to $\pa D$, there is then a unique point $\pi(z) \in \pa D$ that is closest to it, i.e., $\delta(z) = {\rm dist}(z, \pa D) = \vert z - \pi(z) \vert$. Regarding $v$ as a vector at $z$, we will write $v = v_H(\pi(z)) + v_N(\pi(z))$ and abbreviate them as $v_H(z), v_N(z)$ respectively. Let $\cal L_{\psi}(p)$ be the hermitian matrix $\big( \pa^2 \psi / \pa z_{\al} \pa \ov z_{\be} (p) \big)$. The Levi form is
\[
\langle \cal L_{\psi}(p)v, w \rangle = \sum_{\al, \be = 1}^n \frac{\pa^2 \psi}{\pa z_{\al} \pa \ov z_{\be}}(p) v^{\al} \ov w^{\be}
\]
for $p \in D$, $v = (v^1, v^2, \ldots, v^n), w = (w^1, w^2, \ldots, w^n) \in \mbf C^n$, and where $\langle \cdot, \cdot \rangle$ is the standard hermitian inner
product on $\mbf C^n$. When $v = w$, the Levi form will be denoted by $\cal L_{\psi}(p, v)$ for the sake of brevity. Further let $F^R_D(z, v) = \big( ds^2_z(v, v)
\big)^{1/2}$ be the length of the tangent vector $v$ at $z \in D$ in the $\La$--metric.

\begin{thm}
Let $D$ be a $C^{\infty}$-smoothly bounded pseudoconvex domain in $\mbf C^n$ and fix $z_0 \in \pa D$. Suppose that $\psi$ is a $C^{\infty}$-smooth defining function
for $D$. Then for $z \in D$ and $v \in \mbf C^n$,
\begin{enumerate}
\item[(i)] $\lim_{z \ra z_0} F^R_D \big( z, v_N(z) \big) \, \big( -\psi(z) \big) = (2n - 2)^{1/2} \; \big\vert v_N(z_0) \big \vert \, \big\vert \psi(z_0) \big \vert$ and
\item[(ii)] $\lim_{z \ra z_0} F^R_D \big( z, v_H(z) \big) \, \big( -\psi(z) \big)^{1/2} = (2n - 2)^{1/2} \; \Big( \cal L_{\psi} \big( z_0, v_H(z_0) \big) \Big)^{1/2}$
\end{enumerate}
where the limits are uniform for $v$ in a compact subset of $\mbf C^n$.
\end{thm}

\no It will be seen from the proof of this theorem that pseudoconvexity plays no role in deriving these asymptotics; its presence (as far as this statement is
concerned) merely serves to guarantee that the $\La$--metric is well defined. Moreover since $\vert \psi(z) \vert \approx \delta(z)$ for $z$ close to $\pa D$, (i) and (ii)
above provide the asymptotic rate at which the $\La$--metric blows up in the normal and tangential directions. When $D$ is strongly pseudoconvex, there is a
uniform positive lower bound on the Levi form and hence these asymptotics are essentially the same as those obtained by Graham for the Kobayashi
metric on a strongly pseudoconvex domain (see \cite{G}). This suggests that the $\La$-metric and the Kobayashi metric must have similar behaviour even globally
on such domains. Let $d_R(p, q)$ denote the distance in the $\La$-metric between $p, q$ in a given domain $D$. Likewise let $d_K(p, q), d_C(p, q)$ and
$d_B(p, q)$ be the distance between $p, q$ in the Kobayashi, the Carath\'{e}odory and the Bergman metrics respectively on $D$.
\begin{thm}
Let $D$ be a $C^{\infty}$-smoothly bounded strongly pseudoconvex domain in $\mbf C^n$. Then there exists a constant $C \ge 1$ such that
\[
C^{-1} \; d_K(p, q) \le d_R(p, q) \le C \; d_K(p, q)
\]
for all $p, q \in D$. The same holds with $d_K(p, q)$ replaced by $d_C(p, q)$ or $d_B(p, q)$ with a possibly different $C$.
\end{thm}

\no This has three consequences. First this recovers a result from \cite{LY} that the $\La$-metric is complete on strongly pseudoconvex domains. Second, it is
possible to get an expression for $d_R(p, q)$ with a uniformly bounded additive constant using ideas from \cite{BB}. First recall the Carnot-Carath\'{e}odory metric on the
boundary of a smooth strongly pseudoconvex domain $D$ with defining function $\psi$. Call a piecewise $C^1$-smooth curve $\al : [0, 1] \ra \pa D$ {\it horizontal} if for
every $t \in [0, 1]$ for which $\al'(t)$ exists, we have $\al'(t) \in H_{\al(t)}(\pa D)$. The strong pseudoconvexity of $\pa D$ implies that every pair of points $p, q \in
\pa D$ can be connected by a horizontal curve. The length of $\al$ is
\[
l_{\al} = \int_0^1 \cal L_{\psi} \big(\al(t), \al'(t) \big) \, dt
\]
and the Carnot-Carath\'{e}odory distance $d_H(p, q)$ between $p, q \in \pa D$ is the infimum of the lengths $l_{\al}$ for all possible horizontal curves $\al$ that join
$p, q$. Define $g : D \times D \ra \mbf R$ by
\[
g(x, y) = 2 \log \Bigg( \frac{ d_H \big( \pi(x), \pi(y) \big) + \max \big( h(x), h(y) \big) } {\sqrt{ h(x) h(y)} }  \Bigg)
\]
where $h^2(x) = \de(x)$. Note that $\pi$ is apriori well-defined and smooth only near $\pa D$, but a smooth extension to all of $D$ can be chosen and fixed once and for
all to define $g$.

\begin{cor}
Let $D$ be a $C^{\infty}$-smoothly bounded strongly pseudoconvex domain in $\mbf C^n$. Then there are uniform constants $\al > 1$ and $C > 0$ such that
\[
\al^{-1} \, g(p, q) - C \le d_R(p, q) \le \al \, g(p, q) + C
\]
for all $p, q \in D$.
\end{cor}

Third, it also implies that for a strongly pseudoconvex domain $D$, the metric space $(D, d_R)$ is $\delta$-hyperbolic in the sense of Gromov, a
notion which may be briefly described as follows (see \cite{BB} for more details): let $(X, d)$ be a metric space and $I = [a, b] \subset \mbf R$ a compact
interval. Fix $x, y \in X$ and let $\gamma : I \ra X$ be a path that joins $x$ and $y$. Call $\gamma$ a geodesic segment if it is an isometry, i.e., for all
$s, t \in I$, $d(\gamma(s), \gamma(t)) = \vert s - t \vert$. Geodesic segments that join $x$ and $y$ will be denoted by $[x, y]$ inspite of their potential
non-uniqueness. $(X, d)$ is a geodesic space if any pair $x, y \in X$ can be joined by a geodesic segment. A geodesic space $(X, d)$ is called
$\delta$-hyperbolic for some $\delta \ge 0$ if every geodesic triangle $[x, y] \cup [y, z] \cup [z, x]$ in $X$ is $\delta$-thin, i.e.,
\[
{\rm dist} \big( w, [y, z] \cup [z, x] \big) < \delta
\]
for any $w \in [x, y]$. Said differently, there is a $\delta \ge 0$ such that the $\delta$-neighbourhood of any two sides of a given geodesic triangle
contains the third. Therefore in a coarse sense all geodesic triangles are thin and $(X, d)$ behaves like a negatively curved manifold. A definition that
works for more general metric spaces (possibly non-geodesic) is as follows: choose a base point $w \in X$ and define the product of $x, y \in X$ with respect
to $w$ as
\[
(x, y)_w = \frac{1}{2} \big\{ d(x, w) + d(y, w) - d(x, y) \big\}.
\]
Then $X$ is $\delta$-hyperbolic for some $\delta \ge 0$ if
\[
(x, y)_w \ge \min \big\{ (x, z)_w, (z, y)_w \big\} - \delta
\]
for all $x, y, z, w \in X$, which is equivalent to the more symmetric relation
\[
d(x, y) + d(z, w) \le \max \big\{ d(x, z) + d(y, w), d(x, w) + d(y, z) \big\} + 2 \de.
\]
It can be shown that for geodesic metric spaces these definitions are equivalent with possibly a different $\delta \ge 0$.

\medskip

On the other hand the belief that $(D, d_R)$, where $D$ is strongly pseudoconvex, behaves like a negatively curved space
should be further substantiated by estimating the curvatures of the $\La$-metric near $\pa D$ as was done for the Bergman metric in \cite{Kle} using Fefferman's
expansion of the Bergman kernel on smooth strongly pseudoconvex domains. Recall that the holomorphic sectional curvature at $z \in D$ along the direction
$v \in \mbf C^n$ is given by
\[
R(z, v) = \frac{R_{i \ov j k \ov l}(z) v^i \ov v^j v^k \ov v^l}{\big( g_{i \ov j} v^i \ov v^j \big)^2}
\]
where
\[
R_{i \ov j k \ov l} = - \frac{\pa^2 g_{i \ov j}}{\pa z_k \pa \ov z_l}(z) + g^{\nu \ov \mu}(z) \frac{\pa g_{i \ov \mu}}{\pa z_k}(z) \; \frac{\pa g_{\nu \ov j}}{\pa \ov z_l}(z)
\]
where $g^{\nu \ov \mu} = \big( g_{\al \ov \be} \big)^{-1}$ and $g_{\al \ov \be}(z) = \pa^2 \log(-\La) / \pa z_{\al} \pa \ov z_{\be} (z)$ and the standard
convention of summing over all indices that appear once in the upper and lower positions is being followed.

\begin{thm}
Let $D$ be a $C^{2}$-smoothly bounded strongly pseudoconvex domain in $\mbf C^n$. Fix $z_0 \in \pa D$ and $v
\in \mbf C^n$. Then for $z \in D$
\[
R \big(z, v_N(z) \big) \ra -1/(n - 1)
\]
as $z \ra z_0$ along the inner normal to $\pa D$ at $z_0$.
\end{thm}
\no As can be expected the proof of this statement uses the asymptotics of $\La$ from theorem 1.1. However it does not follow directly from the information
provided by theorem 1.1. To explain this heuristically, note that corollary 1.2 shows that
\[
g_{\al \ov \be}(z) \backsim \frac{\pa \psi}{\pa z_{\al}}(z_0) \; \frac{\pa \psi}{\pa \ov z_{\be}}(z_0) \Big/ \big( \psi(z) \big)^2
\]
for $z$ close to $z_0$ and hence
\[
\det \big(g_{\al \ov \be}(z) \big) \backsim \det \Big( \frac{\pa \psi}{\pa z_{\al}}(z_0) \; \frac{\pa \psi}{\pa \ov z_{\be}}(z_0) \Big) \Big / \big( \psi(z) \big)^{2n}
\]
which is apriori indeterminate since the numerator vanishes and $\psi(z) \ra 0$ as $z \ra z_0$. Thus the difficulty is to control $g^{\nu \ov \mu}$. One way
to do this is to show that
\[
\frac{\pa g_{\al \ov \be}}{\pa z_{\gamma}}(z) \backsim \frac{1}{\big( \psi(z) \big)^2}
\]
for $z$ close to $z_0$ and all $1 \le \al \le n - 1$, $1 \le \be, \gamma \le n$ and this can be done along the inner normal at $z_0$. A direct calculation
using the asymptotics of theorem 1.1 yields a rate of blow up for the derivatives of $g_{\al \ov \be}$ that is of the order of $\big( \psi(z) \big)^{-3}$ which is not
enough to control the indeterminacy in $g^{\nu \ov \mu}$. While a similar estimate for the holomorphic curvature
should hold along tangential directions (and indeed along all directions) without any additional assumptions on the way $z$ approaches $z_0$, a stronger claim
about the blow up of the fourth order derivative of $\La$ near the boundary is needed and this is not a direct consequence of theorem
1.1. This point will be discussed further later on in the article.

\medskip

Another remark about the $\La$-metric is motivated by Lu-Qi Keng's theorem which states
that a bounded domain in $\mbf C^n$ with complete Bergman metric and whose holomorphic sectional curvature is equal to a negative constant everywhere must be
biholomorphic to the ball. Using ideas from \cite{CW} it is possible to prove a version of this result for the $\La$-metric.

\begin{thm}
Let $D$ be a $C^{\infty}$-smoothly bounded strongly convex domain in $\mbf C^n$. Suppose that the $\La$-metric on $D$ has constant negative holomorphic sectional
curvature. Then $D$ is biholomorphic to the ball $\mbf B^n$ and the $\La$-metric is proportional to the Bergman metric.
\end{thm}

\no Finally, it is also possible to show the interior stability of the $\La$--metric on a smoothly varying family of domains in $\mbf C^n$.

\begin{thm}
Let $D$ be a $C^{\infty}$-smoothly bounded pseudoconvex domain in $\mbf{C}^{n}$ and let $D_j$ be a sequence of smoothly bounded pseudoconvex domains that converge to $D$ in the
$C^2$-topology. Let $ds^2, ds^2_j$ be the $\La$--metrics on $D$ and $D_j$ respectively. Denote by $R(p, v), R_j(p, v)$ the holomorphic sectional curvatures of $D, D_j$
respectively at $p \in D$ along $v \in \mbf C^n$. Then $ds^2_j \ra ds^2$ uniformly on compact subsets of $D$ and $R_j \ra R$ uniformly on compact subsets of $D \times \mbf
C^n$.
\end{thm}

\medskip

\no {\it Acknowledgements :} Some of the material presented here has benefited from the discussions that the first author had with Professors Kang-Tae Kim and Rasul Shafikov. Many thanks are due to them for their remarks and encouragement. We would also like to thank Harish Seshadri for a timely and very useful clarification regarding Gromov hyperbolic spaces and in particular for pointing out the relevance of \cite{Va} to this work. The second author is indebted to N. Levenberg and H. Yamaguchi for various helpful clarifications.

\section{Properties of the Robin function}

\no To obtain information about $\La(p)$ it would be desirable to first understand a few of its basic properties. All the properties mentioned below were
proved in \cite{Y} and the sole purpose of this exposition is to clarify certain salient features that will be needed later. It should also be noted that
while some of these properties hold on a more general class of domains, it will suffice for us to restrict attention to bounded regular domains.

\begin{lem}\label{Lambda-from-G}
Let $D \subset \mbf C^n$ be a bounded regular domain and fix $p \in D$. Then for every $r > 0$ such that the ball $B(p, r)$ is relatively compact in $D$,
\begin{align}\label{Lambda-from-G-exp}
\Lambda(p) = -\frac{1}{r^{2n-2}}+\frac{1}{\sigma_{2n} r^{2n - 1}} \int_{\partial B(p,r)} G(z,p) \;dS_{z}
\end{align}
where $\sigma_{2n}$ is the surface area of the unit sphere in $\mbf C^n$ and $dS_{z}$ is the surface area element on $\partial B(p,r)$. Moreover $\La < 0$ in $D$.
\end{lem}

\begin{proof}
As noted in (1.1)
\[
G(z, p) = \vert z - p \vert^{-2n + 2} + \La(p) + h_p(z)
\]
where $h_p(z)$ is harmonic on $D$ and $h_p(p) = 0$. For every $r > 0$ as above integrate this over $\pa B(p, r)$ to get
\[
\La(p) = -\frac{1}{r^{2n-2}}+\frac{1}{\sigma_{2n} r^{2n - 1}} \int_{\partial B(p,r)}G(z,p) \;dS_{z} + \frac{1}{\sigma_{2n} r^{2n - 1}} \int_{\pa B(p, r)} h_p(z) \; dS_z.
\]
Since $h_p(z)$ is harmonic in $D$, the last term in the expression above, which is the mean value of $h_p(z)$ on $\pa B(p, r)$ equals $h_p(p) = 0$ and this gives the
desired result.

\medskip

For the other claim, choose $R > 0$ large enough so that $D \subset B = B(0, R)$. For $p \in D$, let $G_D(z, p)$ and $G_B(z, p)$ be the Green functions for $D$ and
$B$ respectively with pole at $p$. Then $G_D(z, p) - G_B(z, p)$ is harmonic on $D$ and equals $-G_B(z, p)$ on $\pa D$.
Hence the maximum principle shows that $G_D(z, p) \le G_B(z, p)$ on $D$. Consequently for all $p \in D$, $\La(p)$ is at most the Robin function for $B$ at $p$ which is
negative.
\end{proof}

\begin{lem}
Let $D \subset \mbf C^n$ be a bounded regular domain. Then $\Lambda$ is real analytic on $D$.
\end{lem}

\begin{proof}
Since $G(z, p)$ is known to be symmetric in $z, p$ so is $H_p(z) = G(z, p) - \vert z - p \vert^{-2n + 2}$. Hence the function $H$ defined on $D \times D$ as $H(z, p) =
H_p(z)$ is harmonic in $z$ as well as $p$. Moreover the maximum principle implies that $H(z, p) \le 0 $ in $D \times D$. Now fix $p_0 \in D$ and let $B = B(p_0, r)$ be a
relatively compact ball in $D$. Applying the Poisson integral formula twice and appealing to Fubini's theorem shows that
\begin{align*}
\Lambda(p) = \frac{1}{(\sigma_{2n}r)^{2}}\iint_{\partial B\times \partial B} H(z, w) \frac{ \big( r^2 - \vert z - p \vert^{2} \big) \big( r^{2}- \vert w - p \vert^{2} \big)}{\vert z- p
\vert^{2n}\vert w- p \vert^{2n}} \; dS_{z} \;dS_{w}
\end{align*}
for all $p \in B$. It follows that $\Lambda(p)$ is real analytic in $B$.
\end{proof}

\begin{lem}
Let $D \subset \mbf C^n$ be a bounded regular domain. Then $-\Lambda$ is an exhaustion function for $D$.
\end{lem}

\begin{proof}
Let $p_0 \in \partial D$ and let $M >0$ be given. Choose a ball $B = B(p_0, r)$ such that $\vert z - p \vert^{-2n + 2} > M$ for all $z, p \in B$. Let
$u_{-M}(z)$ be the harmonic function in $D$ whose boundary values are $-M$ on $B \cap \partial D$ and $0$ on $\partial D \sm B$. Fix $p \in B \cap D$
and consider the function $s(z)$ in $D$ defined by
\[
 s(z)=
 \begin{cases}
 u_{-M}(z) - H_p(z),  &\text{if $z \not= p$}\\
 u_{-M}(z) - \La(p),  &\text{if $z = p$}.
\end{cases}
\]
The function $s(z)$ is harmonic in $D$. To see what the boundary values of $s(z)$ are, fix $a \in B \cap \pa D$. Then
\[
s(z) \ra -M + \vert a - p \vert^{-2n + 2} > 0
\]
as $z \ra a$. On the other hand if $a \in \pa D \sm B$, then
\[
s(z) \ra -H_p(a) = \vert a - p \vert^{-2n + 2} > 0
\]
as $z \ra a$.
It follows from the maximum principle that $s(z) \geq 0$ in $D$ and in particular $s(p) \geq 0$. Hence $u_{-M}(p) \geq \Lambda(p)$ for all $ p \in D \cap B$.
Consequently,
\[
\limsup_{p\rightarrow  p_{0}}\Lambda( p )\leq \limsup_{p\rightarrow  p_{0}}u_{-M}(p) = -M,
\]
which means that $\lim_{p\rightarrow p_{0}}\Lambda(p)= -\infty$.
\end{proof}


\section{Boundary behaviour of $\La$}

\no The main goal of this section is to prove theorem 1.1. The basic strategy is to blow up a neighbourhood of $z_0 \in \Gamma$ by means of affine maps $T_j$
associated with a sequence $z_j \ra z_0$ to get a half space as was done in \cite{LY}. The affine maps produce a sequence of scaled domains $D_j = T_j(D)$
and we study the family of Green functions associated to these domains, but the first step is to localise the problem near $z_0$. For this fix a
neighbourhood $U$ of $z_0 \in \Gamma$ and a smooth local defining function $\psi$ for $U \cap \Gamma$ so that $U \cap D = \{z \in \mbf C^n : \psi < 0\}$.

\begin{prop}
There exists a neighbourhood $V$ of $z_0$ compactly contained in $U$ and a constant $C > 0$ depending only on $U$ such that
\[
G_{U \cap D}(z, p) \le G_D(z, p) \le G_{U \cap D}(z, p) + C
\]
for all $z, p \in V$.
\end{prop}

\begin{proof}
Fix $R > 0$ so that $B = B(z_0, 2R)$ is compactly contained in $U$ and let $V = B(z_0, R)$. The first inequality is known and follows from the inclusion $U
\cap D \subset D$. Indeed, the harmonic function $G_{U \cap D}(z, p) - G_D(z, p)$ on $U \cap D$ is at most $-G_D(z, p) \le 0$ on the boundary of $U \cap D$.
Therefore the maximum principle shows that $G_{U \cap D}(z, p) \le G_D(z, p)$ for all $z, p \in U$ and in particular if $z, p \in V$.

\medskip

For the upper estimate let $\Gamma_1 = B \cap \pa D$, $\Gamma_2 = \pa D \sm B$ and $\Gamma_3 = \pa B \cap D$. For a bounded regular domain $\Om \subset \mbf C^n$, let $d
\mu^{\Om}_z$ denote the harmonic measure on $\pa \Om$ at $z \in \Om$ and recall that
\[
G_{\Om}(z, \xi) = \vert z - \xi \vert^{-2n + 2} - \int_{\pa \Om} \vert t - \xi \vert^{-2n + 2} \; d\mu^{\Om}_z
\]
for $z, \xi \in \Om$. Noting that $\Gamma_1$ is a common component of the boundaries of $B \cap D$ and $D$ and appealing to this  representation of the Green function it
follows that
\begin{multline*}
G_D(z, p) - G_{U \cap D}(z, p) \le G_D(z, p) - G_{B \cap D}(z, p)\\
  = - \Big( \int_{\Gamma_1} \vert t - p \vert^{-2n + 2} \; d \mu^D_z - \int_{\Gamma_1} \vert t - p \vert^{-2n + 2} \; d \mu^{B \cap D}_z \Big) - \int_{\Gamma_2} \vert t -
p\vert^{-2n + 2} \; d \mu^D_z + \int_{\Gamma_3} \vert t - p \vert^{-2n + 2} \; d \mu^{B \cap D}_z
\end{multline*}
for $z, p \in V$. Since $d \mu^{B \cap D}_z \le d \mu^D_z$ the first term in the brackets is non-negative and so is the integral over $\Gamma_2$ since $d \mu^D_z \ge 0$
for all $z$. Finally if $t \in \Gamma_3$ and $p \in V$ it follows that $\vert t - p \vert^{-2n + 2} \le R^{-2n + 2}$ for all $z \in V$. It follows that $G_D(z, p) - G_{U
\cap D}(z, p) \le R^{-2n + 2}$ for all $z, p \in V$.
\end{proof}

\no Let $\{z_j\}$ be a sequence in $D$ that converges to $z_0$ and we may assume without loss of generality that all $z_j \in U \cap D$. Consider the affine maps
\[
T_j(z) = (z - z_j) / (-\psi(z_j))
\]
that blow up any fixed neighbourhood of $z_0$ in the sense that for any compact $K \subset \mbf C^n$ there is a $j_0$ such that $K \subset T_j(U)$ for all $j \ge j_0$.
Let $D_j = T_j(D)$ and $(U \cap D)_j = T_j(U \cap D)$. Note that $T_j(z_j) = 0$ and hence $0 \in D_j$ for all $j$. The defining function for $D_j$ in $(U \cap D)_j$ is
\begin{align*}
\psi \circ T^{-1}_j(z) & = \psi \Big(z_j + z \big( -\psi(z_j) \big) \Big)\\
                    & = \psi(z_j) +   2 \Re \Big( \sum_{\al = 1}^n \frac{\pa \psi}{\pa z_{\al}}(z_j) z_{\al} \Big) \big(-\psi(z_j) \big) + \big(\psi(z_j)\big)^2
\,O(1).
\end{align*}
Then
\[
\ti \psi_j(z) = \psi \circ T^{-1}_j(z) \big / (-\psi(z_j)) = -1 +  2 \Re \Big( \sum_{\al = 1}^n \frac{\pa \psi}{\pa z_{\al}}(z_j) z_{\al} \Big) + \big(-\psi(z_j)\big)
\,O(1)
\]
is again a defining function for $D_j$ in $(U \cap D)_j$ and in the limit it can be seen that these defining functions converge to
\[
\psi_{\infty}(z) = -1 +  2 \Re \Big( \sum_{\al = 1}^n \frac{\pa \psi}{\pa z_{\al}}(z_0) z_{\al} \Big).
\]
in the $C^2$ topology on every compact set in $\mbf C^n$. In particular this implies that the domains $D_j$ converge to the half space
\[
\cal H = \bigg \{ z \in \mbf C^n : 2 \Re \Big( \sum_{\al = 1}^n \frac{\pa \psi}{\pa z_{\al}}(z_0) z_{\al} \Big) - 1 < 0 \bigg \}
\]
in the Hausdorff sense. To see this let $K \subset \cal H$ be compact. Then $\psi_{\infty}(K) < 0$ and since $K \subset T_j(U)$ for large $j$ it follows that $\ti
\psi_j(K)  < 0$ for all large $j$. Conversely if $K$ is compactly contained in $D_j$ then $\ti \psi_j(K) < -c < 0$ for some uniform $c = c(K) > 0$. Passing to the limit it
follows that $\psi_{\infty}(K) \le -c < 0$ which exactly means that $K$ is compactly contained in $\cal H$. Exactly the same argument shows that $(V \cap D)_j$ converges
to $\cal H$ as well in the Hausdorff sense where $V$ is as in proposition 3.1.

\medskip

Let $G_j$ be the Green function for $D_j$ and let $\La_j$ be the associated Robin function. Likewise let $G_{\cal H}$ be the Green
function for the half space $\cal H$ and $\La_{\cal H}$ the corresponding Robin function. For brevity let $\tilde D = U \cap D$ where $U$ is as above and $\ti D_j =
T_j(U \cap D)$. Finally let $\ti G_j, \ti \La_j$ be the Green function and the associated Robin function for $\ti D_j$. Using (1.1) it can be seen that
\begin{equation}
G_j \big( T_j(z), T_j(p) \big) = G_D(z, p) \, \big( \psi(z_j) \big)^{2n - 2}
\end{equation}
and that
\[
\La_j \big( T_j(p) \big) = \La(p) \, \big( \psi(z_j) \big)^{2n - 2}
\]
which can be rewritten as
\begin{equation}
\La_j(p) = \La \big( z_j - p \, \psi(z_j) \big) \, \big( \psi(z_j) \big)^{2n - 2}
\end{equation}
for all $j$ and $p \in D_j$.

\begin{prop}
For every $p \in \cal H$, $\big\{ G_j(z, p) \big\}$ has a subsequence that converges uniformly on compact subsets of $\cal H \sm \{p\}$ to a function $\ti G(z, p)$
which is harmonic on $\cal H \sm \{p\}$.
\end{prop}

\begin{proof}
First assume that $p = 0 \in D_j$ for all $j$ which implies that $G_j(z, 0)$ is well defined on $D_j \sm \{0\}$. Let $K$ be a compact subset of $\cal H \sm \{0\}$. Then $K
\subset T_j(V \cap D) \subset \ti D_j \subset D_j$ for large $j$. For $z \in K$ and $j$ large, proposition 3.1
shows that
\[
G_{\ti D}\big( T^{-1}_j(z), z_j) \big) \le G_D \big(T^{-1}_j(z), z_j \big) \le G_{\ti D}\big( T^{-1}_j(z), z_j) \big) + C
\]
for some uniform constant $C > 0$ and this is equivalent to
\begin{equation}
\ti G_j(z, 0) \le G_j(z, 0) \le \ti G_j(z, 0) + C \, \big( \psi(z_j) \big)^{2n - 2}
\end{equation}
by (3.1). Now $\ti G_j(z, 0)$ is harmonic on $\ti D_j \sm \{0\}$ and satisfies $0 \le \ti G_j(z, 0) \le \vert z \vert^{-2n + 2}$ there by the maximum principle. In
particular, $\big\{\ti G_j(z, 0) \big\}$ is uniformly bounded on $K$ and hence it admits a subsequence that converges uniformly on compact subsets of $\cal H \sm \{0\}$ to
a function $\ti G$ that is harmonic on $\cal H \sm \{0\}$. By (3.3) the same is therefore true of $G_j(z, 0)$.

\medskip

In general if $p \in \cal H$ and $K \subset \cal H \sm \{p\}$ is compact, then both $K$ and $p$ are contained in $T_j(V \cap D) \subset \ti D_j \subset D_j$ for $j$ large
so that $G_j(z, p)$ is well defined on $K$. As before
\[
G_{\ti D}\big( T^{-1}_j(z), T^{-1}_j(p) \big) \le G_D \big(T^{-1}_j(z), T^{-1}_j(p) \big) \le G_{\ti D}\big( T^{-1}_j(z), T^{-1}_j(p)) \big) + C
\]
for some uniform $C > 0$ and $j$ large. This implies that
\begin{equation}
\ti G_j(z, p) \le G_j(z, 0) \le \ti G_j(z, p) + C \, \big( \psi(z_j) \big)^{2n - 2}
\end{equation}
as before. Since $\ti G_j(z, p)$ is harmonic on $\ti D_j \sm \{p\}$ and satisfies $0 \le \ti G_j(z, p) \le \vert z - p \vert^{-2n + 2}$ there by the maximum principle, it
follows that $\big\{ \ti G_j(z, p) \big\}$ is uniformly bounded on $K$. In particular there is a subsequence that converges uniformly on compact subsets of
$\cal H \sm \{p\}$ to $\ti G(z, p)$ that is harmonic on $\cal H \sm \{p\}$. By (3.4) the same holds for $\big\{ G_j(z, p) \big\}$.
\end{proof}

\begin{prop}
$\ti G(z, p)$ is the Green function for $\cal H$ with pole at $p$, i.e., $\ti G(z, p) = G_{\cal H}(z, p)$.
\end{prop}

\no The proof of this requires a quantitative understanding of the behaviour of $\ti G(z, p)$ near $\pa \cal H$ and for this an auxiliary step from \cite{Kr}
will be needed. Once this has been proved it will follow that every convergent subsequence of $G_j(z, p)$ has a unique limit, namely $G_{\cal H}(z, p)$. Hence $G_j(z,
p) \ra G_{\cal H}(z, p)$ uniformly on compact subsets of $\cal H \sm \{p\}$. A point $x \in \mbf R^m$ will be written as $x = ('x, x_m) \in \mbf R^{m - 1} \times \mbf
R$.

\begin{lem}
Let $B \subset \mbf R^m$ be the unit ball. Let $\ti B = B(('0, 1), 1), \Om = B \cap \ti B^c$ and $\tau = \pa B \sm \ti B$. Then there is a harmonic function
$H$ on a neighbourhood of $\ov \Om$ which is non-negative on $\ov \Om$ and a constant $c = c(B) > 0$ such that

\begin{enumerate}
\item[(i)] $H(0) = 0$

\item[(ii)] $H(x) > 0$ if $x \in \ov \Om \sm \{0\}$

\item[(iii)] $H(x) \ge 1$ for all $x \in \tau$

\item[(iv)] $H('0, -t) \le c t$ for all $0 \le t \le 1$.
\end{enumerate}
\end{lem}

\begin{proof}
The Kelvin transform of $f(x) = x_m$ with respect to $\pa \ti B$ is given by
\[
f^*(x) = \frac{f(x^*)}{\vert x - ('0, 1) \vert^{m - 2}}
\]
for $x \in \mbf R^m \sm \{('0, 1)\}$ and
\[
x^* = \frac{x - ('0, 1)}{\vert x - ('0, 1) \vert^2} + ('0, 1)
\]
which is the inversion of $x$ with respect to $\pa \ti B$. Thus
\[
f^*(x) = \frac{(x_m - 1) + \vert x - ('0, 1) \vert^2}{\vert x - ('0, 1) \vert^m}
\]
and since $f$ is harmonic on $\mbf R^m$, it follows that $f^*$ is harmonic on $\mbf R^m \sm \{('0, 1)\}$. Evidently $f^*(0) = 0$. Let $x \in \ov \Om \sm \{0\}$. Then
$\vert x - ('0, 1) \vert \ge 1$. Two cases arise -- first, if $x_m > 0$ then
\[
(x_m - 1) + \vert x - ('0, 1) \vert^2 > - 1 + 1 = 0
\]
and hence $f^*(x) > 0$. On the other hand if $x_m < 0$ then
\[
(x_m - 1) + \vert x - ('0, 1) \vert^2 = x_1^2 + x_2^2 + \ldots + x_{m - 1}^2 + x_m(x_m - 1) > 0
\]
and hence $f^*(x) > 0$ again. Finally if $x_m = 0$ then the above expression is the sum of squares of the first $m - 1$ components of $x$ and this is positive since $x
\not= 0$. Thus $f^*(x) > 0$. It follows that there is a constant $c > 0$ such that $c f^* \ge 1$ on $\ov \tau$. Let $H(x) = c f^*(x)$. By construction $H$ satisfies (i),
(ii) and (iii). For (iv) note that if $0 \le t \le 1$ then
\[
H('0, -t) = ct/(1 + t)^{m - 1} \le ct.
\]
\end{proof}

\no In general if $L : \mbf R^m \ra \mbf R^m$ is an affine transformation of the form $L(x) = \al Ax + b$ where $\al > 0$ and $A$ is orthogonal and $\Om' = L(\Om)$,
then $\ti H(x) = H \circ L^{-1}$ is harmonic in a neighbourhood of $\ov \Om'$ and satisfies $\ti H(b) = 0$, $\ti H(x) > 0$ on $\ov \Om' \sm \{b\}$, $\ti H(x) \ge 1$
on $L(\tau)$ and $\ti H(L('0, -t)) \le ct$ for all $0 \le t \le 1$. Now fix $z_0 \in \pa \cal H \sm \{\infty\}$ and a neighbourhood $U$ of $z_0$. Since the defining
functions for $U \cap \pa \ti D_j$ converge to that of $U \cap \pa \cal H$ in the $C^2(\ov U)$ topology, the implicit function theorem shows that for $j$ large there is a
uniform $R > 0$ such that if
$r < R$ and $q_j \in U \cap \pa \ti D_j$, there is a ball of radius $r$ around $c = c(q_j)$ that lies outside $\pa \ti D_j$ and whose closure touches $\pa \ti
D_j$ only at $q_j$.

\begin{prop}
Let $z_0 \in \pa \cal H \sm \{\infty\}$ and $K \subset \cal H$ compact. Then there exists a neighbourhood $U$ of $z_0$ disjoint from $K$ and a constant $C = C(z_0, K)
> 0$ such that
\[
\ti G(z, p) \le C \, \delta(z) \,\vert z - p \vert^{-2n + 1}
\]
where $z \in U \cap \cal H$, $p \in K$ and $\delta(z)$ is the distance of $z$ to $\pa \cal H$.
\end{prop}

\begin{proof}
Fix a neighbourhood $U$ of $z_0$ disjoint from $K$ and let $R > 0$ be as above. Fix $p \in K$ and without loss of generality assume that
\[
R < \big( {\rm dist}(z_0, K) \big) \big/ 2.
\]
Choose $s > 0$ sufficiently small that
\[
\de(z) < \big( R \, \vert z - p \vert \big) \big/ \big( 4 \, \vert z_0 - p \vert \big)
\]
for all $z \in B(z_0, s)$ and $p \in K$. By shrinking $s$ if needed we may assume that $\pi(z) \in U$ whenever $z \in B(z_0, s) \cap D$ and that $s < {\rm dist}(z_0,
K)$. Now fix $z \in B(z_0, s) \cap \cal H$ and define
\[
\al = \big( R \, \vert z - p \vert \big) \big/ \big( 2 \, \vert z_0 - p \vert \big) < \vert z - p \vert/4
\]
and note that
\[
2 \, \de(z) < \al \le \frac{R \big( \vert z - z_0 \vert + \vert z_0 - p \vert \big)}{2\vert z_0 - p \vert} < R
\]
since $\vert z - z_0 \vert \le s < {\rm dist}(z_0, K)$. Let $\pi_j(z) \in U \cap \pa \ti D_j$ realise the distance between $z$ and $U \cap \pa \ti D_j$ and let
$\de_j(z) = {\rm dist}(z, U \cap \pa \ti D_j) = \vert \pi_j(z) - z \vert$ and note that $\de_j(z) \ra \de(z)$ as $j \ra \infty$. Define $\Om_j
= B(\pi_j(z), \al) \cap B^c(c_j, \al)$ where $c_j = c_j(\pi_j(z))$. If $w \in \pa \Om_j \cap \ti D_j$ then
\[
\vert z - w \vert \le \vert z - \pi_j(z) \vert + \vert \pi_j(z) - w \vert = \de_j(z) + \al < 2 \al
\]
and therefore
\begin{align*}
\vert p - w \vert & \ge \vert p - z \vert - \vert z - w \vert\\
                  & \ge \vert p - z \vert - 2 \al \ge \vert p - z \vert/2.
\end{align*}
This shows that
\begin{equation}
\ti G_j(w, p) \le \vert w - p \vert^{-2n + 2} \le 2^{2n - 2} \vert z - p \vert^{-2n + 2}
\end{equation}
for $w \in \pa \Om_j \cap \ti D_j$. Let $H_j$ be the harmonic function on $\Om_j$ whose existence is guaranteed by lemma 3.4 and the remark after it. Since $H_j(w)
\ge 1$ on $\pa \Om_j \cap \ti D_j$ it follows from (3.5) that
\[
\ti G_j(w, p) \le H_j(w)\, 2^{2n - 2} \vert z - p \vert^{-2n + 2}
\]
for $w \in \pa \Om_j \cap \ti D_j$. But if $w \in \Om_j \cap \pa \ti D_j$ then
\[
0 = \ti G_j(w, p) \le  H_j(w)\, 2^{2n - 2} \vert z - p \vert^{-2n + 2}
\]
and hence the maximum principle shows that
\[
\ti G_j(w, p) \le  H_j(w)\, 2^{2n - 2} \vert z - p \vert^{-2n + 2}
\]
for all $w \in \Om_j \cap \ti D_j$. Let $n_j(\pi_j(z))$ be the outward real normal to $\pa \ti D_j$ at $\pi_j(z)$ and define $z_j = \pi_j(z) - \de(z) \,
n_j(\pi_j(z))$. Since $z_j \ra z$ and $\ti D_j$ converge to $\cal H$ in the Hausdorff sense it follows that $z_j \in D_j$ for all large $j$. Hence
\begin{equation}
\ti G_j(z_j, p) \le  H_j(z_j)\, 2^{2n - 2} \vert z - p \vert^{-2n + 2}
\end{equation}
for all large $j$. But property (iv) of $H$ from lemma 3.4 shows that $H_j(z_j) \le c \,\de(z) / \al$ and with this (3.6) becomes
\[
\ti G_j(z_j, p) \le (c \,\de(z) / \al) \, 2^{2n - 2} \vert z - p \vert^{-2n + 2}
\]
which is the same as
\[
\ti G_j(z_j, p) \le C \,\de(z)\, \vert z - p \vert^{-2n + 1}
\]
after using the definition of $\al$ with $C = (c/R) 2^{2n - 1} \big( {\rm dist}(z_0, K) \big)^{-1} > 0$. It remains to observe that $\ti G_j(z_j, p) \ra \ti G(z, p)$ which
completes the proof.
\end{proof}

\no To continue with the proof of proposition 3.3, fix $p \in \cal H$. Since $\ti G(z, p)$ is already known to be harmonic on $\cal H \sm \{p\}$ it suffices to show that
$\ti G(z, p) - \vert z - p \vert^{-2n + 2}$ is harmonic near $p$ and that $\ti G(z, p) \ra 0$ as $z \ra \pa \cal H$. Fix a small ball $B = B(p, r) \subset \cal H$ and note
that $H_j(z, p) =  G_j(z, p) - \vert z - p \vert^{-2n + 2}$ is harmonic on a neighbourhood of $\ov B$ for all $j$. Moreover
\[
H_j(z, p) \ra \ti G(z, p) - \vert z - p \vert^{-2n + 2}
\]
uniformly on $\pa B$ and hence on $\ov B$ by the maximum principle. It follows that $\ti G(z, p) - \vert z - p \vert^{-2n + 2}$ is harmonic near $p$. For the other claim,
note that
\[
0 \le \ti G(z, p) \le \vert z - p \vert^{-2n + 2}
\]
for all $z \in \cal H \sm \{p\}$ since the same holds for all $G_j(z, p)$. For a given $\ep > 0$ it is therefore possible to choose a large ball $B(p, R)$ such that $\ti
G(z, p) < \ep$ for $z \in \cal H \sm B(p, R)$. On the other hand there is a finite cover of $\pa \cal H \cap \ov B(p, R)$ by open balls in which the estimate
\[
\ti G(z, p) \lesssim \de(z) \, \vert z - p \vert^{-2n + 1}
\]
of proposition 3.5 holds. Combining these observations it follows that if $z$ is close enough to $\pa \cal H$ then $\ti G(z, p) < \ep$ and this completes the proof
of proposition 3.3.

\begin{prop}
$G_j(z, p) - \vert z - p \vert^{-2n + 2}$ converges to $G_{\cal H}(z, p) - \vert z - p \vert^{-2n + 2}$ uniformly on compact subsets of $\cal H \times \cal H$.
\end{prop}

\begin{proof}
By a linear change of coordinates we may assume that $\cal H = \{z \in \mbf C^n : \Re z_n < 0\}$. First consider $(z_0, p_0) \in \cal H \times \cal H$ off the diagonal,
i.e., $z_0 \not= p_0$. Let $U_0, V_0 \subset \cal H$ be neighbourhoods of $z_0, p_0$ respectively
that have disjoint closures. For a given $\ep > 0$, choose $R > 0$ large such that
\[
2 \vert z - p \vert^{-2n + 2} < \ep
\]
for all $\vert z \vert > R$ and $p \in V_0$. Define
\[
\Om(\tau, \eta) = \big\{ z \in \mbf C^n : \vert z_j \vert \le 2 \tau \;\, {\rm for}\;\, 1 \le j \le n - 1, \vert \Im z_n \vert \le 2 \tau, -2 \tau \le \Re z_n \le -\eta
\big\}
\]
for $\tau, \eta > 0$. Note that $\Om(\tau, \eta) \subset \cal H$ and that $U_0, V_0 \subset \Om(\tau, \eta)$ by taking $\tau \gg 1$ and $\eta > 0$ small enough. Moreover
for a fixed $\tau, \eta$, $\Om(\tau, \eta) \subset D_j$ for all large $j$. Therefore $G_j(z, p)$ is well defined on $\Om(\tau, \eta)$ for each $p \in V_0$ and $j$ large.
By the maximum principle
\[
\big \vert G_j(z, p) - G_{\cal H}(z, p) \big \vert < 2 \vert z - p \vert^{-2n + 2} < \ep
\]
for $z \in \pa \Om(R, \eta) \sm B(0, R)$ and $p \in V_0$. For $\eta > 0$ small, there is a finite cover of $\pa \Om(R, \eta) \cap B(0, R)$ in which
\[
G_j(z, p) \lesssim \delta(z) \, \vert z - p \vert^{-2n + 1} \;\, {\rm and} \;\, G_{\cal H} \lesssim \delta(z) \, \vert z - p \vert^{-2n + 1}
\]
for $j$ large, each $p \in V_0$ and $z \in \pa \Om(R, \eta) \cap B(0, R)$. This is a consequence of proposition 3.5. Therefore
\[
\big \vert G_j(z, p) - G_{\cal H}(z, p) \big \vert \lesssim \delta(z) \, \vert z - p \vert^{-2n + 1} = \eta \vert z - p \vert^{-2n + 1} < \ep
\]
for $z \in \pa \Om(R, \eta) \cap B(0, R)$ and $p \in V_0$ if $\eta > 0$ is small enough. Now for each $p \in V_0$, the function $G_j(z, p) - G_{\cal H}(z, p)$ is harmonic
in $\Om(R, \eta)$ and in modulus is at most $\ep$ on $\pa \Om(R, \eta)$. By the maximum principle
\[
\big \vert G_j(z, p) - G_{\cal H}(z, p) \big \vert < \ep
\]
for each $p \in V_0$ and $z \in U_0 \subset \Om(R, \eta)$ in particular. Hence
\[
G_j(z, p) - \vert z - p \vert^{-2n + 2} \ra G_{\cal H}(z, p) - \vert z - p \vert^{-2n + 2}
\]
uniformly on compact subsets of $\cal H \times \cal H \sm (\rm diagonal)$.

\medskip

Now fix $(z_0, z_0) \in \cal H \times \cal H$. Let $\cal N$ be a small neighbourhood of $(z_0, z_0)$ in the diagonal in $\cal H \times \cal H$. Then there is a uniform
$\rho > 0$ such that the family $\cal F$ of horizontal plaques parametrised by
\[
\phi_z(\z) = (z + \rho \zeta, z)
\]
where $\z = (\z_1, \z_2, \ldots, \z_n) \in \Delta^n$ and $z \in \cal N$ contains a neighbourhood of $(z_0, z_0)$ along the diagonal possibly smaller than $\cal
N$. By shrinking $\cal N$ if needed it follows that the boundary $\pa \cal F$ of this family of horizontal plaques is disjoint from the diagonal. The argument above shows
that
\[
G_j(z, p) - \vert z - p \vert^{-2n + 2} \ra G_{\cal H}(z, p) - \vert z - p \vert^{-2n + 2}
\]
uniformly on $\pa \cal F$ and hence the maximum principle shows that the convergence is uniform even around $(z_0, z_0)$ and this completes the proof.
\end{proof}

\no {\it Proof of Theorem 1.1:} From (3.2) it follows that
\[
D^{A \ov B} \La_j(0) = (-1)^{\vert A \vert + \vert B \vert} \big( D^{A \ov B} \La(z_j) \big) \big( \psi(z_j) \big)^{2n - 2 + \vert A \vert + \vert B \vert}
\]
where the derivatives are with respect to $p$. Now fix a compactly contained ball $U = B(0, r) \subset \cal H$ and note that $U \subset D_j$ for all $j$ large. For each
such $j$ and $p \in U$, lemma 2.2 shows that
\[
\La_j(p) = \frac{1}{(\sigma_{2n}r)^{2}} \iint_{\partial U \times \partial U} H_j(z, w) \frac{ \big( r^2 - \vert z - p \vert^{2} \big) \big( r^{2}- \vert w - p \vert^{2} \big)}{\vert z- p
\vert^{2n}\vert w - p \vert^{2n}} \; dS_{z} \;dS_{w}
\]
where $H_j(z, w) = G_j(z, w) - \vert z - w \vert^{-2n + 2}$. Differentiating with respect to $p$ under the integral sign shows that
\[
D^{A \ov B} \La_j(p) = \frac{1}{(\sigma_{2n}r)^{2}} \iint_{\partial U \times \partial U} H_j(z, w) D^{A \ov B} \bigg( \frac{r^2 - \vert z - p \vert^{2}}{\vert z - p
\vert^{2n}} \bigg)  \bigg( \frac{r^2 - \vert w - p \vert^{2}}{\vert w - p \vert^{2n}} \bigg)  \; dS_{z} \;dS_{w}.
\]
Since $H_j(z, w)$ converges to $H_{\infty}(z, w) = G_{\cal H}(z, w) - \vert z - w \vert^{-2n + 2}$ uniformly on $\pa U \times \pa U$, the integral above converges to
\[
\frac{1}{(\sigma_{2n}r)^{2}} \iint_{\partial U \times \partial U} H_{\infty}(z, w) D^{A \ov B} \bigg( \frac{r^2 - \vert z - p \vert^{2}}{\vert z - p
\vert^{2n}} \bigg)  \bigg( \frac{r^2 - \vert w - p \vert^{2}}{\vert w - p \vert^{2n}} \bigg)  \; dS_{z} \;dS_{w} = D^{A \ov B} \La_{\cal H}(p).
\]
Setting $p = 0$ shows that
\[
(-1)^{\vert A \vert + \vert B \vert} \big( D^{A \ov B} \La(z_j) \big) \big( \psi(z_j) \big)^{2n - 2 + \vert A \vert + \vert B \vert} \ra D^{A \ov B} \La_{\cal H}(0)
\]
as $z_j \ra z_0$ which finishes the proof.


\section{Comparison with the Kobayashi metric}

\no Let $D$ be a $C^{\infty}$-smoothly bounded domain and $\psi$ a $C^{\infty}$-smooth defining function for $D$. i.e., $D = \{ \psi < 0\}$, $\pa D = \{\psi =
0\}$ and the gradient of $\psi$ does not vanish on $\pa D$. Fix $z_0 \in \pa D$. Define
\begin{equation}
\la(z) = \La(z) \big( \psi(z) \big)^{2n - 2}
\end{equation}
for $z \in D$. In what follows, the standard convention of denoting derivatives by suitable subscripts will be followed. For example
$\psi_{\al} = \pa \psi/ \pa z_{\al}$, $\psi_{\al \ov \be} = \pa^2 \psi/ \pa z_{\al} \pa \ov z_{\be}$ etc. By theorem 1.1 it follows that
\[
\la(z) = \La(z) \big(\psi(z) \big)^{2n - 2} \ra - \big \vert \nabla \psi(z_0) \big \vert^{2n - 2}
\]
as $z \ra z_0$ which implies that $\la$ is continuous on $\ov D$. A stronger result was proved in \cite{LY}, namely that $\la$ is $C^2$-smooth on $\ov D$. They in fact
conjecture that $\la$ must be $C^{\infty}$ smooth on $\ov D$.

\medskip

It must be noted that the boundary behaviour of $\La$ and its second order derivatives was obtained as a consequence of the $C^2$ smoothness of $\la$ on $\ov D$
in \cite{LY}. On the other hand theorem 1.1 does not involve $\la$ and so it is tempting to see if the boundary smoothness of $\la$ can be obtained using the
information there. As seen above, this approach does yield the continuity of $\la$ on $D$. However there is a problem when derivatives of $\la$ are considered. Indeed,
\[
\la_{\al} = \La_{\al} \,\psi^{2n - 2} + (2n - 2) \La \, \psi^{2n - 3} \, \psi_{\al}
\]
for $1 \le \al \le n$. By theorem 1.1 it follows that
\[
\La_{\al}(z) \, \big( \psi(z) \big)^{2n - 1} \ra (2n - 2) \psi_{\al}(z_0)\, \big \vert \nabla \psi(z_0) \big \vert^{2n - 2}
\]
and hence $\la_{\al}(z) \, \psi(z) \ra 0$ as $z \ra z_0$ which does not apriori show that $\la_{\al}$ is continuous at $z_0$. $D$ will now be assumed to be
pseudoconvex so that the $\La$--metric is well defined and $\langle \cdot, \cdot \rangle$ as before will denote the standard hermitian inner product on $\mbf C^n$. To
begin with note that a straightforward calculation using (4.1) implies that

\begin{lem}
Let $\la$ be as above. For $1 \le \al, \be \le n$
\begin{enumerate}
\item[(i)] $\La_{\al} =  \psi^{-2n + 2} \big( \la_{\al} - (2n - 2)\la \,\psi_{\al}\, \psi^{-1} \big)$,
\item[(ii)] $\La_{\al \ov \be} = \psi^{-2n + 2} \big( \la_{\al \ov \be} - (2n - 2)(\la_{\al}\, \psi_{\ov \be} + \la_{\ov \be}\, \psi_{\al}) \,\psi^{-1} \\
\no \hspace*{100pt} + (2n - 2)(2n - 1) \la \, \psi_{\al} \, \psi_{\ov \be} \, \psi^{-2} - (2n - 2) \la \, \psi_{\al \ov \be} \, \psi^{-1} \big)$.
\end{enumerate}
\end{lem}

\begin{prop}
Let $D, \psi$ be as above. Then there is a neighbourhood $\cal N$ of $\pa D$ such that
\[
\Big\vert \big\langle v_H(z), \ov \pa \psi(z) \big\rangle \Big\vert = \Big \vert \sum_{\al = 1}^n \psi_{\al}(z) \, v^{\al}_H(z) \Big \vert \lesssim \delta(z)
\]
for all $z \in \cal N \cap D$ and vectors $v$ of unit length.
\end{prop}

\begin{proof}
Choose a tubular neighbourhood $\cal N$ of $\pa D$ so small that the distance $\delta(z)$ of $z$ to $\pa D$ is realised by a unique point $\pi(z) \in \pa D$ and such
that all derivatives of $\psi$ up to second order are uniformly bounded on it. For $v \in \mbf C^n$ and $z \in \cal N \cap D$ note that
\[
\big\langle v_H(z), \ov \pa \psi \big(\pi(z) \big) \big\rangle = 0
\]
and thus
\[
\big\langle v_H(z), \ov \pa \psi(z) \big\rangle = \big\langle v_H(z), \ov \pa \psi(z) - \ov \pa \psi \big( \pi(z) \big) \big \rangle.
\]
Since $\vert z - \pi(z) \vert = \delta(z)$ it follows that
\[
\Big \vert \big \langle v_H(z), \ov \pa \psi(z) \big \rangle \Big\vert \lesssim \delta(z)
\]
for $z \in \cal N \cap D$ and $\vert v \vert = 1$.
\end{proof}

\no {\it Proof of Theorem 1.3:} For $v \in \mbf C^n$ regarded as a tangent vector at $z \in D$ we have
\begin{equation}
\big( F^R_D(z, v) \big)^2 = \sum_{\al , \be = 1}^n \frac{\pa^2 \log(- \La)}{\pa z_{\al} \pa \ov z_{\be}} (z) v^{\al} \ov v^{\be} = A - \vert B \vert^2
\end{equation}
where
\[
A = \La^{-1} \sum_{\al, \be = 1}^n \La_{\al \ov \be} v^{\al} \ov v^{\be} = \La^{-1} \cal L_{\La}(z, v) \;\; {\rm and} \;\; B = \La^{-1} \sum_{\al = 1}^n
\La_{\al} v^{\al} = \La^{-1} \langle v, \ov \pa \La \rangle.
\]
>From $\la = \La \psi^{-2n + 2}$ and lemma 4.1 it is seen that
\begin{multline*}
A = \la^{-1} \cal L_{\la}(z, v) - 2(2n - 2)(\la \, \psi)^{-1} \,\Re \big( \langle v, \ov \pa \la \rangle \, \langle \pa \psi, \ov
v \rangle \big)\\
+ (2n - 2)(2n - 1) \psi^{-2} \big \vert \langle v, \ov \pa \psi \rangle \big \vert^2 - (2n - 2) \psi^{-1} \cal L_{\psi}(z, v)
\end{multline*}
and
\[
B = \la^{-1} \langle v, \ov \pa \la \rangle - (2n - 2) \, \psi^{-1} \, \langle v, \ov \pa \psi \rangle.
\]
Now fix a compact $S \subset \mbf C^n$. If $v$ varies in $S$ and $z$ is close to $\pa D$ then
\begin{equation}
v_N = v_N(z) = \frac{\big \langle v, \ov \pa \psi \big( \pi(z) \big) \big \rangle}{\big \vert \ov \pa \psi \big( \pi(z) \big) \big \vert^2} \, \ov \pa \psi \big( \pi(z) \big)
\end{equation}
also varies compactly. Since $\la$ is $C^2$ on $\ov D$ and non-zero there, both $\la^{-1} \langle v_N, \ov \pa \la \rangle$ and $\la^{-1} \, i \pa \ov \pa \la(v_N, v_N)$
are uniformly bounded for all $v \in S$ and $z$ near $\pa D$. Also (4.3) shows that
\[
\big \vert \langle v_N, \ov \pa \psi \rangle \big \vert \ra \big \vert v_N(z_0) \big \vert \, \big \vert \ov \pa \psi(z_0) \big \vert = \big \vert v_N(z_0) \big \vert \, \big \vert \pa \psi(z_0) \big \vert
\]
and
\[
\cal L_{\psi}(z, v_N) \ra \cal L_{\psi} \big( z_0, v_N(z_0) \big)
\]
as $z \ra z_0$, the convergence being uniform for $v \in S$. It follows that
\[
\big( -\psi(z) \big)^2 \, A \ra (2n - 2)(2n - 1) \vert v_N(z_0) \vert^2 \, \vert \pa \psi(z_0) \vert^2
\]
and
\[
\big( -\psi(z) \big) B \ra -(2n - 2) \big \vert v_N(z_0) \big \vert \, \big \vert \pa \psi(z_0) \big \vert
\]
as $z \ra z_0$. Consequently (4.2) shows that
\begin{align*}
\big( -\psi(z) \big)^2 \Big( F^R_D(z, v_N(z)\big) \Big)^2 & = \big( -\psi(z) \big)^2 \, A - \big( -\psi(z) \big)^2 \, \vert B \vert^2\\
& \ra (2n - 2) \big \vert v_N(z_0) \big \vert^2 \, \big \vert \pa \psi(z_0) \big \vert^2
\end{align*}
as $z \ra z_0$ which proves the first assertion of this theorem.

\medskip

For $v_H = v_H(z) = v - v_N(z)$ the terms $\la^{-1} \langle v_H, \ov \pa \la \rangle$ and $\la^{-1} \cal L_{\la}(z, v_H)$ are uniformly bounded for $v
\in S$ and $z$ near $\pa D$ as before. By the previous lemma
\[
\big \vert \psi^{-1} \langle v_H, \ov \pa \psi \rangle \big \vert \lesssim 1
\]
if $z$ is close enough to $\pa D$ since $\vert \psi(z) \vert \approx \delta(z)$. This shows that $B$ is bounded near $z_0$. Also
\[
\cal L_{\psi}(z, v_H) \ra \cal L_{\psi} \big( z_0, v_H(z_0) \big)
\]
as $z \ra z_0$, the convergence being uniform in $v \in S$. It follows that
\[
\big( -\psi(z) \big) \, A \ra (2n - 2)\, \cal L_{\psi}\big( z_0, v_H(z_0) \big)
\]
and
\[
\big( -\psi(z) \big) \, \vert B \vert^2 \ra 0
\]
as $z \ra z_0$. Consequently (4.2) shows that
\begin{align*}
\big( -\psi(z) \big) \Big( F^R_D \big( z, v_H(z) \big) \Big)^2 & = \big( -\psi(z) \big) \, A - \big( -\psi(z) \big) \, \vert B \vert^2\\
& \ra (2n - 2) \, \cal L_{\psi} \big( z_0, v_H(z_0) \big)
\end{align*}
as $z \ra z_0$ which proves the second assertion and completes the proof of this theorem.

\medskip

\no For the proof of theorem 1.4, which is a statement about the global comparibility of the $\La$--metric and the Kobayashi metric it is essential to know the behaviour of
$F^R_D(z, v)$ near $\pa D$ for an arbitrary vector $v$.

\begin{prop}
Let $D \subset \mbf C^n$ be a bounded strongly pseudoconvex domain with $C^{\infty}$-smooth boundary and fix $z_0 \in \pa D$. Suppose that $\psi$ is a $C^{\infty}$-smooth
defining function for $\pa D$. Then for $z \in D$ and $v \in \mbf C^n$
\begin{equation}
\lim_{z \ra z_0} \frac{ \big( F^R_D(z, v) \big)^2 }{ \big( F^R_D(z, v_N) \big)^2 + \big( F^R_D(z, v_H) \big)^2 } = 1
\end{equation}
where the limit is uniform for $v$ in a compact subset of $\mbf C^n$ and as usual the decomposition $v = v_H + v_N$ is taken at $\pi(z) \in \pa D$. In particular there
exists a constant $C > 1$ such that
\begin{equation}
\frac{1}{C} \left(\frac{\vert v_N \vert^2}{\de^2(z)} + \frac{ \cal L_{\psi} \big(\pi(z), v_H \big)}{\de(z)} \right)^{1/2} \le F^R_D(z, v) \le C \,\left( \frac{\vert v_N
\vert^2}{\de^2(z)} + \frac{ \cal L_{\psi} \big( \pi(z), v_H \big)}{\de(z)} \right)^{1/2}
\end{equation}
for all $z$ sufficiently close to $\pa D$ and all $v \in \mbf C^n$.
\end{prop}

\begin{proof}
>From $v = v_H + v_N$ for any $z \in D$ sufficiently close to $\pa D$ and the fact that the $\La$--metric is Hermitian it follows that
\[
\big( F^R_D(z, v) \big)^2 = \big( F^R_D(z, v_H) \big)^2 + \big( F^R_D(z, v_N) \big)^2 + 2 \, \Re \langle v_H, v_N \rangle_R
\]
where $\langle X, Y \rangle_R = \sum_{\al, \be = 1}^n g_{\al \ov \be}(z) X^{\al} \ov Y^{\be}$ for vectors $X = (X^1, X^2, \ldots, X^n)$ and $Y = (Y^1, Y^2, \ldots, Y^n)$.
Since $g_{\al \ov \be}(z) = \pa^2 \log(-\La) / \pa z_{\al} \pa \ov z_{\be}(z)$ we have that
\[
\langle v_H, v_N \rangle_R = \La^{-1} \big \langle \cal L_{\La}(z) v_H, v_N \big \rangle - \La^{-1} \langle v_H, \ov \pa \La \rangle \cdot \La^{-1}
\langle \ov v_N, \pa \La \rangle.
\]
There are two cases that need to be considered. First suppose that $v_N(z_0) \not= 0$. From theorem 1.1 it is known that

\begin{itemize}
\item $\psi^{2n - 2} \La \ra - \big \vert \nabla \psi(z_0) \big \vert^{2n - 2}$,

\item $\psi^{2n - 1} \La_{\al} \ra (2n - 2) \psi_{\al}(z_0) \big \vert \nabla \psi(z_0) \big \vert^{2n - 2}$ and

\item $\psi^{2n} \La_{\al \ov \be} \ra -(2n - 1)(2n - 2) \psi_{\al}(z_0) \psi_{\ov \be}(z_0) \big \vert \nabla \psi(z_0) \big \vert^{2n - 2}$
\end{itemize}
as $z \ra z_0$. As a result
\begin{align*}
\psi^2 \langle v_H, v_N \rangle_R & \ra (2n - 1)(2n - 2)  \sum_{\al, \be = 1}^n \psi_{\al}(z_0) \psi_{\ov \be}(z_0) v_H^{\al}(z_0) \ov v_N^{\be}(z_0) - (2n - 2)^2 \big \langle v_H(z_0), \ov \pa \psi(z_0) \big \rangle \, \big \langle \ov v_N(z_0), \pa \psi(z_0) \big \rangle\\
& = (2n - 1)(2n - 2) \big \langle v_H(z_0), \ov \pa \psi(z_0) \big \rangle \, \big \langle \ov v_N(z_0), \pa 
\psi(z_0) \big \rangle - (2n - 2)^2  \big \langle v_H(z_0), \ov \pa \psi(z_0) \big \rangle \, \big \langle \ov v_N(z_0), \pa \psi(z_0) \big \rangle\\
& = (2n - 2) \big \langle v_H(z_0), \ov \pa \psi(z_0) \big \rangle \, \big \langle \ov v_N(z_0), \pa 
\psi(z_0) \big \rangle = 0
\end{align*}
as $z \ra z_0$ where the last equality holds since $\big \langle v_H(z_0), \ov \pa \psi(z_0) \big \rangle = 0$. On the other hand it is seen from theorem 1.3 that
\[
\Big( \big(-\psi(z) \big) \, F^R_D(z, v_N) \Big)^2 + \Big( \big(-\psi(z) \big) \, F^R_D(z, v_H) \Big)^2 \ra (2n - 2) \big \vert v_N(z_0) \big \vert^2 \, \big \vert \nabla \psi(z_0) \big \vert^2 > 0
\]
as $z \ra z_0$. There is no contribution from the horizontal component $v_H$ since $F^R_D(z, v_H) \backsim \big(-\psi(z) \big)^{-1/2}$ near $\pa D$ by theorem 1.3 and hence
\[
\big( (-\psi(z)) \, F^R_D(z, v_H) \big)^2 \ra 0
\]
as $z \ra z_0$. Putting all this together we see that
\[
\frac{ \langle v_H, v_N \rangle_R }{ \big( F^R_D(z, v_N) \big)^2 + \big( F^R_D(z, v_H) \big)^2 } \ra 0
\]
as $z \ra z_0$ which implies (4.4) in case $v_N(z_0) \not= 0$.

\medskip

Now suppose that $v_N(z_0) = 0$ which exactly means that $v$ lies in the complex tangent space to $\pa D$ at $z_0$ and that $v_N(z) \ra 0$ as $z \ra z_0$. As above, each of the terms $\langle v_H, v_N \rangle_R, F^R_D(z, v_N)$ and $F^R_D(z, v_H)$ will be separately discussed, starting with $F^R_D(z, v_N)$. From (4.2) we have
\[
\big( -\psi(z) \big) \big( F^R_D(z, v_N) \big)^2 = \big( -\psi(z) \big) \, A - \big( -\psi(z) \big) \, \vert B \vert^2
\]
where
\begin{multline*}
A = \la^{-1} \cal L_{\la}(z, v_N) - 2(2n - 2)(\la \, \psi)^{-1} \,\Re \big( \langle v_N, \ov \pa \la \rangle \, \langle \pa \psi, \ov
v_N \rangle \big)\\
+ (2n - 2)(2n - 1) \psi^{-2} \big \vert \langle v_N, \ov \pa \psi \rangle \big \vert^2 - (2n - 2) \psi^{-1} \cal L_{\psi}(z, v_N)
\end{multline*}
and
\[
B = \la^{-1} \langle v_N, \ov \pa \la \rangle - (2n - 2) \, \psi^{-1} \, \langle v_N, \ov \pa \psi \rangle.
\]
Let $I, II, III, IV$ denote the four terms in the expression for $A$. Since $\la$ is $C^2$ on $\ov D$ and non-zero there, it follows that $\big( -\psi(z) \big) \, I \ra 0$ while
$\big( -\psi(z) \big) \, II \ra 0$ and $\big( -\psi(z) \big) \, IV \ra 0$ as $v_N(z) \ra 0$. Note that $\big( -\psi(z) \big) \, III = (2n - 2)(2n - 1){\psi}^{-1} \big \vert \langle v_N, \ov \pa \psi
\rangle \big \vert^2$. Since $v_N(z_0) = 0$
\begin{align}
\big \langle v_N(z), \ov \pa \psi(z) \big \rangle & = \big \langle v_N(z), \ov \pa \psi(z) \big \rangle - \big \langle v_N(z_0), \ov \pa \psi(z_0) \big \rangle \\
                              & = \big \langle v_N(z), \ov \pa \psi(z) - \ov \pa \psi(z_0) \big \rangle + \big \langle v_N(z) - v_N(z_0), \ov \pa \psi(z_0) \big \rangle.
\end{align}
It is evident that
\[
\big \vert \ov \pa \psi(z) - \ov \pa \psi(z_0) \big \vert \lesssim \de(z)
\]
while (4.2) combined with the fact that $\psi$ has non-vanishing gradient near $\pa D$ shows that
\begin{align*}
v_N(z) - v_N(z_0) & \approx  \Big \langle v, \ov \pa \psi \big( \pi(z) \big) \Big \rangle \, \ov \pa \psi \big( \pi(z) \big) - \big \langle v, \ov \pa \psi(z_0) \big \rangle \, \ov \pa \psi(z_0)\\
& \approx \Big \langle v, \ov \pa \psi \big( \pi(z) \big) \Big \rangle \, \Big( \ov \pa \psi \big( \pi(z) \big) - \ov \pa \psi(z_0) \Big) + \Big \langle v, \ov \pa \psi \big( \pi(z) \big) - \ov \pa \psi(z_0) \Big \rangle \, \ov \pa \psi(z_0).
\end{align*}
Hence it follows that
\[
\big \vert v_N(z) - v_N(z_0) \big \vert \lesssim \de(z).
\]
Since $\vert \psi(z) \vert \approx \de(z)$, the observations made above collectively show that
\begin{equation}
\psi^{-1} \Big \vert \big \langle v_N(z), \ov \pa \psi(z) \big \rangle \Big \vert^2 \ra 0
\end{equation}
as $z \ra z_0$. As a result $\big( -\psi(z) \big) \, A \ra 0$. On the other hand note that
\[
\big( -\psi(z) \big) \vert B \vert^2 = \big \vert (-\psi)^{1/2} \, \la^{-1} \langle v_N, \ov \pa \la \rangle + (2n - 2)(-\psi)^{-1/2} \langle v_N, \ov \pa \psi \rangle \big
\vert^2.
\]
The first term evidently goes to zero while the second term also goes to zero thanks to (4.6). As a result $\big( -\psi(z) \big) \big( F^R_D(z, v_N) \big)^2 \ra 0$ as $z \ra
z_0$. By theorem 1.1, it is known that
\[
\big( -\psi(z) \big) \big( F^R_D(z, v_H) \big)^2 \ra (2n - 2) \, \cal L_{\psi}(z_0, v) > 0
\]
as $z \ra z_0$ where the positivity of the Levi form follows from the strong pseudoconvexity of $D$ and the assumption that $v$ is complex tangential to $\pa D$ at
$z_0$. The other term is $\langle v_H, v_N \rangle_R$; for this we write
\[
\big( -\psi(z) \big) \, \langle v_H, v_N \rangle_R = \big( -\psi(z) \big) S - \big( -\psi(z) \big) T
\]
where
\begin{multline*}
S = \la^{-1} \big \langle \cal L_{\la}(z) v_H, v_N \big \rangle - (2n - 2) \psi^{-1} \la^{-1} \big( \langle v_H, \ov \pa \la \rangle \, \langle \ov \pa \psi, v_N \rangle + \langle v_H, \ov \pa \psi \rangle \, \langle \ov \pa \la, v_N \rangle \big)\\
+ (2n - 1)(2n - 2) \psi^{-2} \la \, \langle v_H, \ov \pa \psi \rangle \, \langle \ov \pa \psi, v_N \rangle - (2n - 2) \psi^{-1} \la \big \langle \cal L_{\psi}(z) v_H, v_N \big \rangle
\end{multline*}
and
\[
T = \la^{-1} \big(  \langle v_H, \ov \pa \la \rangle - (2n - 2) \psi^{-1} \la \langle v_H, \ov \pa \psi \rangle \big) \, \la^{-1} \big( \langle \ov \pa \la, v_N
\rangle - (2n - 2) \psi^{-1} \la \langle \ov \pa \psi, v_N \rangle \big).
\]
Let $I, II, III, IV$ be the four terms in the expression for $S$. Since $\la$ is $C^2$ on $\ov D$ and non-zero there, it follows that $\big( -\psi(z) \big) I \ra 0$ while
$\big( -\psi(z) \big) II \ra 0$ and $\big( -\psi(z) \big) IV \ra 0$ as $v_N(z) \ra 0$. Note that $\big( -\psi(z) \big) III = (2n - 1)(2n - 2) \psi^{-1} \la  \langle v_H, \ov \pa \psi \rangle \,
\langle \ov \pa \psi, v_N \rangle$. By proposition 4.2 it follows that $\big \vert \langle v_H, \ov \pa \psi \rangle \big \vert \lesssim \delta(z)$ near $z_0$ which implies that
\[
\big \vert \psi^{-1} \langle v_H, \ov \pa \psi \rangle \big \vert \lesssim 1
\]
near $z_0$. The other factor $\langle \ov \pa \psi, v_N \rangle \ra 0$ as $z \ra z_0$ since $v_N(z) \ra 0$. Hence $\big( -\psi(z) \big) S \ra 0$.

\medskip

On the other hand note that
\[
\big( -\psi(z) \big) T = \la^{-1} \big(  \psi \langle v_H, \ov \pa \la \rangle - (2n - 2) \la \langle v_H, \ov \pa \psi \rangle \big) \, \la^{-1} \big( \langle \ov \pa
\la, v_N \rangle - (2n - 2) \psi^{-1} \la \langle \ov \pa \psi, v_N \rangle \big).
\]
Evidently $\psi \langle v_H, \ov \pa \la \rangle \ra 0$ since $\psi(z_0) = 0$ and $\la$ is smooth near $z_0$. Proposition 4.2 shows that $\big \vert \langle v_H, \ov \pa \psi \rangle \big \vert \lesssim \delta(z)$ and hence the first factor above vanishes at $z_0$. The smoothness of $\la$ again implies that $\langle \ov \pa \la, v_N \rangle \ra 0$ as $z \ra z_0$. Finally, the calculations leading up to (4.6) show that $\big \vert \langle \ov \pa \psi, v_N \rangle \big \vert \lesssim \delta(z)$ and this in turn implies that the remaining term, i.e., $(2n - 2) \psi^{-1} \la \langle \ov \pa \psi, v_N \rangle$ is bounded near $z_0$. All this together shows that $\big( -\psi(z) \big) T \ra 0$ as $z \ra z_0$ and this is sufficient to conclude that (4.4) holds even when $v_N(z_0) = 0$. The estimates in (4.5) are now a consequence of theorem 1.3.

\end{proof}

\no {\it Proof of theorem 1.4:} It is known (see \cite{Gr2} for example) that the Kobayashi metric on a smoothly bounded strongly pseudoconvex domain $D \subset \mbf C^n$
behaves as
\begin{equation}
\big( F^K_D(z, v) \big)^2 \approx \big( F^K_D(z, v_N) \big)^2 + \big( F^K_D(z, v_H) \big)^2 \approx \frac{\vert v_N(z) \vert^2}{\de^2(z)} + \frac{\vert v_H(z)
\vert^2}{\de(z)}
\end{equation}
where $z \in D$ is close to $\pa D$, $v$ is a tangent vector at $z$ and the decomposition $v = v_N(z) + v_H(z)$ is taken at $\pi(z)$ as usual. Using the homogeneity of
these metrics in the vector variable, proposition 4.3 shows that
\begin{equation}
F^R_D(z, v) \approx F^K_D(z, v)
\end{equation}
uniformly for all $z$ close to $\pa D$ and $v \in \mbf C^n$. On compact subsets of $D$, these metrics are comparable to the euclidean metric. Therefore the length of any
piecewise differentiable path joining $p, q \in D$ is uniformly comparable in these metrics amd hence $d_K(p, q) \approx d_R(p, q)$ for all $p, q \in D$. Since the
Carath\'{e}odory and the Bergman metric have the same behaviour as (4.7) on strongly pseudoconvex domains, theorem 1.4 follows.

\medskip

Consequently, the $\La$--metric on a smoothly bounded strongly pseudoconvex domain is complete. For the proof of corollary 1.5, observe that the following finer estimate
holds for the Kobayashi metric on a smooth strongly pseudoconvex domain with defining function $\psi$, i.e., for $z \in D$ close to $\pa D$ and $v \in \mbf C^n$,
\begin{multline*}
\big( 1 - C \de^{1/2}(z) \big)^2 \left( \frac{\vert v_N \vert^2}{4 \de^2(z)} + (1 - \eta) \frac{\cal L_{\psi} \big( \pi(z), v_H \big)}{\de(z)} \right) \le \big( F^K_D(z, v) \big)^2 \\
\hfill  \le \big( 1 + C \de^{1/2}(z) \big)^2 \left( \frac{\vert v_N \vert^2}{4 \de^2(z)} + (1 + \eta) \frac{\cal L_{\psi}\big( \pi(z), v_H \big)}{\de(z)} \right)
\end{multline*}
where $\eta > 0$ and $C > 0$ are uniform constants and the decomposition $v = v_N + v_H$ is taken at $\pi(z)$. Since $F^R_D(z, v) \approx F^K_D(z, v)$ on such a $D$, it
follows that there exists $\al > 1$ such that
\begin{multline*}
\al^{-1} \, \big( 1 - C \de^{1/2}(z) \big)^2 \left( \frac{\vert v_N \vert^2}{4 \de^2(z)} + (1 - \eta) \frac{\cal L_{\psi}\big( \pi(z), v_H \big)}{\de(z)} \right) \le \big( F^R_D(z, v)
\big)^2 \\
\hfill  \le \al \, \big( 1 + C \de^{1/2}(z) \big)^2 \left( \frac{\vert v_N \vert^2}{4 \de^2(z)} + (1 + \eta) \frac{\cal L_{\psi}\big( \pi(z), v_H \big)}{\de(z)} \right)
\end{multline*}
for all $z \in D$ close to $\pa D$ and vectors $v \in \mbf C^n$. Theorem 1.1 of \cite{BB} shows that these estimates can be integrated to yield
\[
\al^{-1} \; g(p, q) - C \le d_R(p, q) \le \al \; g(p, q) + C
\]
for all $p, q \in D$ and some uniform $C > 0$.

\medskip

To show that $(D, d_R)$ is $\de$-hyperbolic, observe that the identity map
\[
i : (D, d_R) \ra (D, d_K)
\]
is Lipschitz thanks to theorem 1.4. Now given $p, q \in D$ and a continuous path $\gamma : [0, 1] \ra D$ joining them, the length of $\gamma$ in the Kobayashi metric is
\[
l(\gamma) = \sup_{\cal P} \big\{ d_K(\gamma(t_{i - 1}), \gamma(t_i) \big\}
\]
where the supremum is taken over all possible partitions $\cal P : 0 = t_0 < t_1 < \ldots < t_n = 1$. Then
\[
\tilde d(p, q) = \inf_{\gamma} l(\gamma)
\]
where the infimum is taken over all continuous paths $\gamma$ joining $p, q$ is the induced metric which evidently satisfies $d_K(p, q) \le \ti d(p, q)$ for all $p, q \in
D$. If $f : \Delta \ra D$ is holomorphic then it can be checked that
\[
\tilde d \big( f(a), f(b) \big) \le d_{hyp}(a, b)
\]
where $a, b \in \Delta$ and $d_{hyp}$ is the usual hyperbolic metric on $\Delta$, i.e., $\tilde d$ is distance decreasing for holomorphic maps from the disc. However it
is known that $d_K$ is the largest distance with this property and hence $\ti d \le d_K$ on $D$. It follows that $\ti d = d_K$ which means that $d_K$ is intrinsic. The
same property holds (by definition) for $d_R$ since it is hermitian. By \cite{BB} it is known that $(D, d_K)$ is $\de$-hyperbolic. Theorem 3.18 of \cite{Va} now shows
that the same is true of $(D, d_R)$.


\section{Calculation of the holomorphic sectional curvature along normal directions}
\no In this section we prove theorem 1.6. First we show that the $\Lambda$-metric is invariant under certain transformations of $\mathbb{C}^{n}$, namely translations, unitary rotations and dialations.

\begin{lem}
The $\Lambda$-metric is invariant under an affine transformation of $\mathbb{C}^{n}$ of the form
\[
 f(z)=a(Az)+b
\]
where $a \in \mathbb{C} \setminus \{0\}$, $b \in \mathbb{C}^{n}$ and $A$ is an $n \times n$ unitary matrix.
\end{lem}

\begin{proof}
Let $D \subset \mathbb{C}^{n}$ be a domain for which the $\Lambda$-metric is well defined and let $D^{\prime} = f(D)$. We have to prove that
\begin{align}\label{Lambda-metric-is-invariant}
F^{R}_{D}(z, v) = F^{R}_{D^{\prime}}(f(z), df(z)v)
\end{align}
for $z \in D$ and $v \in \mathbb{C}^{n}$. Note that
\[ 
\Lambda_{D^{\prime}}(w) = |a|^{-2n+2} \Lambda_{D}(z)
\]
where $w = f(z)$. This gives
\[
\log(-\Lambda_{D^{\prime}}(w)) = \log(-\Lambda_{D}(z)) + \log  |a|^{-2n+2}.
\]
Differentiating this with respect to $\overline{z}_{\beta}$, and then by $z_{\alpha}$ we obtain
\[
\sum_{l,m} \bigg \{ \frac {\partial^2 \log(-\Lambda_{D^{\prime}} ) } {\partial w_{l} \partial \overline{w}_{m}} (w) \bigg\} \frac{\partial f_{l}}{\partial z_{\alpha}}(z) \frac{\partial\overline{f}_{m}}{\partial \overline{z}_{\beta}}(z)
= \frac {\partial^{2} \log(-\Lambda_{D})} {\partial z_{\alpha} \partial \overline{z}_{\beta}} (z)
\]
Therefore
\[
\sum_{\alpha,\beta} \frac {\partial^{2} \log(-\Lambda_{D})} {\partial z_{\alpha} \partial \overline{z}_{\beta}} (z) v_{\alpha} \overline{v}_{\beta} = \sum_{l,m}  \frac {\partial^2 \log(-\Lambda_{D^{\prime}} ) } {\partial w_{l} \partial \overline{w}_{m}} (w) \Bigg\{ \sum_{\alpha} \bigg( \frac{\partial f_{l}}{\partial z_{\alpha}}(z)v_{\alpha} \bigg) \sum_{\beta} \overline{ \bigg(\frac{\partial f_{m}}{\partial z_{\beta}}(z) v_{\beta} \bigg)} \Bigg\}
\]
i.e.,
\[
\sum_{\alpha,\beta} \frac {\partial^{2} \log(-\Lambda_{D})} {\partial z_{\alpha} \partial \overline{z}_{\beta}} (z) v_{\alpha} \overline{v}_{\beta} = \sum_{l,m}  \frac {\partial^2 \log(-\Lambda_{D^{\prime}} ) } {\partial w_{l} \partial \overline{w}_{m}} (w)\big( df(z)v \big)_{l} \big( \overline{ df(z)v } \big)_{m}
\]
which is the same as (\ref{Lambda-metric-is-invariant}).
\end{proof}

\medskip

\no Let us now consider a $C^{2}$-smooth strongly pseudoconvex domain $D$ in $\mathbb{C}^{n}$. Fix a point $z_{0} \in \partial D$ and suppose that $\psi$ is a $C^{2}$-smooth defining function for $D$ with $|\nabla \psi(z_{0})| = 1$. Since translations and unitary rotations are isometries for the $\Lambda$-metric, the curvature of this metric remains unchanged under transformations of these kinds. Therefore, without loss of generality we will assume that $z_{0} = 0$ and $\nabla \psi (0) = (0, \ldots, 0, 1)$. Let $\mathcal{H}$ be the half space defined in theorem 1.1, i.e.,
\[
\mathcal{H} = \bigg\{ z \in \mathbb{C}^{n} : 2 \Re \Big(\sum_{i=1}^{n} \psi_{i}(0) z_{i} \Big) -1 < 0 \bigg\}.
\]
The Robin function for $\mathcal{H}$ is given by
\[ \Lambda_{\mathcal{H}}(z) = \frac{-\sum_{i=1}^{n} \big| \psi_{i}(0) \big|^{2n-2}} {\Big(2 \Re \big(\sum_{i=1}^{n} \psi_{i}(0) z_{i} \big) -1\Big)^{2n-2}} = - \bigg( 2 \Re \Big(\sum_{i=1}^{n} \psi_{i}(0) z_{i} \Big) -1 \bigg)^{-2n+2}. \]
An application of theorem 1.1 gives the following boundary behaviour of the components of the $\Lambda$-metric.
\begin{lem}\label{lim-g}
For $z \in D$ the functions $g_{\alpha \overline{\beta}}(z)$ satisfy
\begin{enumerate}
\item [(i)] $\lim_{z \rightarrow 0} g_{\alpha \overline{\beta}} (z) \big(\psi(z)\big)^{2} = (2n-2) \psi_{\alpha}(0) \psi_{\overline{\beta}}(0)$,

\item [(ii)] $\lim_{z \rightarrow 0} \frac {\partial g_{\alpha \overline{\beta}}} {\partial z_{\gamma}}(z) \big(\psi(z)\big)^{3} = -2(2n-2) \psi_{\alpha}(0) \psi_{\overline{\beta}}(p) \psi_{\gamma}(0)$ and

\item [(iii)] $\lim_{z \rightarrow 0} \frac{\partial^{2} g_{\alpha\overline{\beta}}} {\partial z_{\gamma} \partial \overline{z}_{\delta}}(z) \big(\psi(z)\big)^{4} = 6(2n-2) \psi_{\alpha}(p) \psi_{\overline{\beta}}(0) \psi_{\gamma}(0) \psi_{\overline{\delta}}(0)$
\end{enumerate}
for $1 \leq \alpha, \beta, \gamma, \delta \leq n$.
\end{lem}
\begin{proof}
We note that
\begin{itemize}
\item $\Lambda_{\mathcal{H}}(0) = -1$,

\item $(\Lambda_{\mathcal{H}})_{a}(0) = -(2n-2)\psi_{a}(0)$,

\item $(\Lambda_{\mathcal{H}})_{ab}(0) = -(2n-2)(2n-1) \psi_{a}(0) \psi_{b}(0)$,

\item $(\Lambda_{\mathcal{H}})_{abc}(0) = -(2n-2)(2n-1)(2n) \psi_{a}(0) \psi_{b}(0) \psi_{c}(0)$ and

\item $(\Lambda_{\mathcal{H}})_{abcd}(0) = -(2n-2)(2n-1)(2n)(2n+1) \psi_{a}(0) \psi_{b}(0) \psi_{c}(0) \psi_{d}(0)$
\end{itemize}
where the indices $a$, $b$, $c$, $d$ refer to either holomorphic or conjugate holomorphic derivatives. Hence by theorem 1.1 we get
\begin{itemize}
\item $\Lambda(z) \big(\psi(z))^{2n-2} \rightarrow -1$,

\item $\Lambda_{a}(z) \big(\psi(z) \big)^{2n-1} \rightarrow (2n-2) \psi_{a}(0)$,

\item $\Lambda_{ab}(z) \big(\psi(z) \big)^{2n} \rightarrow -(2n-2)(2n-1) \psi_{a}(0) \psi_{b}(0)$,

\item $\Lambda_{abc}(z) \big(\psi(z) \big)^{2n+1} \rightarrow -(2n-2)(2n-1)(2n) \psi_{a}(0) \psi_{b}(0) \psi_{c}(0)$ and

\item $\Lambda_{abcd}(z) \big(\psi(z) \big)^{2n+2} \rightarrow -(2n-2)(2n-1)(2n)(2n+1) \psi_{a}(0) \psi_{b}(0) \psi_{c}(0) \psi_{d}(0)$.
\end{itemize}
Now
\begin{align}\label{g-alpha-beta}
g_{\alpha \overline{\beta}} = \frac {\partial^{2} \log(-\Lambda)} {\partial z_{\alpha} \partial \overline{z}_{\beta}} = \frac {\Lambda_{\alpha \overline{\beta}}} {\Lambda} - \frac {\Lambda_{\alpha} \Lambda_{\overline{\beta}}} {\Lambda^{2}}.
\end{align}
Multiplying both sides of this equation by $\psi^{2}$, we get
\begin{align*}
g_{\alpha \overline{\beta}} \psi^{2} = \frac {\Lambda_{\alpha \overline{\beta}} \psi^{2n}} {\Lambda \psi^{2n-2}} - \frac {(\Lambda_{\alpha} \psi^{2n-1}) (\Lambda_{\overline{\beta}} \psi^{2n-1})} {(\Lambda \psi^{2n-2})^{2}}.
\end{align*}
It follows that 
\[
\lim_{z \rightarrow 0} g_{\alpha \overline{\beta}} (z) \left(\psi(z)\right)^{2} = (2n-2) \psi_{\alpha}(0) \psi_{\overline{\beta}}(0)
\]
which is (i).

\medskip

Differentiating (\ref{g-alpha-beta}) with respect to $z_{\gamma}$
\begin{align}\label{der-g-alpha-beta}
\frac {\partial g_{\alpha \overline{\beta}}} {\partial z_{\gamma}} = \frac {\Lambda_{\alpha \overline{\beta} \gamma}} {\Lambda} - \left(\frac {\Lambda_{\alpha \overline{\beta}} \Lambda_{\gamma}} {\Lambda^{2}} + \frac{\Lambda_{\alpha \gamma} \Lambda_{\overline{\beta}}} {\Lambda^{2}} + \frac{\Lambda_{\overline{\beta} \gamma} \Lambda_{\alpha}} {\Lambda^{2}} \right) + \frac{2 \Lambda_{\alpha} \Lambda_{\overline{\beta}} \Lambda_{\gamma}} {\Lambda^{3}}.
\end{align}
Multiplying both sides of this equation by $\psi^{3}$, we get
\begin{align*}
\frac {\partial g_{\alpha \overline{\beta}}} {\partial z_{\gamma}} \psi^{3} =  & \frac {\Lambda_{\alpha \overline{\beta} \gamma} \psi^{2n+1}} {\Lambda \psi^{2n-2}} - \left(\frac {(\Lambda_{\alpha \overline{\beta}} \psi^{2n}) (\Lambda_{\gamma} \psi^{2n-1})} {(\Lambda \psi^{2n-2})^{2}} + \frac{(\Lambda_{\alpha \gamma} \psi^{2n}) (\Lambda_{\overline{\beta}}\psi^{2n-1})} {(\Lambda \psi^{2n-2})^{2}} + \frac{(\Lambda_{\overline{\beta} \gamma} \psi^{2n}) (\Lambda_{\alpha} \psi^{2n-1})} {(\Lambda \psi^{2n-2})^{2}} \right)\\
& \hspace{2cm}  + \frac{2 (\Lambda_{\alpha} \psi^{2n-1}) (\Lambda_{\overline{\beta}} \psi^{2n-1}) (\Lambda_{\gamma} \psi^{2n-1})} {(\Lambda \psi^{2n-2})^{3}}.
\end{align*}
It follows that 
\[
\lim_{z \rightarrow 0} \frac {\partial g_{\alpha \overline{\beta}}} {\partial z_{\gamma}}(z) \psi(z)^{3} = -2(2n-2) \psi_{\alpha}(0) \psi_{\overline{\beta}}(p) \psi_{\gamma}(0)
\]
which is (ii).

\medskip

Differentiating (\ref{der-g-alpha-beta}) with respect to $\overline{z}_{\delta}$ we obtain
\begin{align*}
\frac{\partial^{2}g_{\alpha \overline{\beta} }}{\partial z_{\gamma}\partial\overline{z}_{\overline{\delta} }}=
\frac{\Lambda_{\alpha \overline{\beta} \gamma\overline{\delta} }}{\Lambda}
-\left(\frac{\Lambda_{\alpha \overline{\beta} \gamma}\Lambda_{\overline{\delta} }}{\Lambda^{2}}
+\frac{\Lambda_{\alpha \overline{\beta} \overline{\delta} }\Lambda_{\gamma}}{\Lambda^{2}}
+\frac{\Lambda_{\alpha \gamma\overline{\delta} }\Lambda_{\overline{\beta} }}{\Lambda^{2}}
+\frac{\Lambda_{\overline{\beta} \gamma\overline{\delta} }\Lambda_{\alpha }}{\Lambda^{2}}\right)
-\left(\frac{\Lambda_{\alpha \overline{\beta} }\Lambda_{\gamma\overline{\delta} }}{\Lambda^{2}}
+\frac{\Lambda_{\alpha \gamma}\Lambda_{\overline{\beta} \overline{\delta} }}{\Lambda^{2}}
+\frac{\Lambda_{\alpha \overline{\delta} }\Lambda_{\overline{\beta} \gamma}}{\Lambda^{2}}\right)\\
+2\left(\frac{\Lambda_{\alpha \overline{\beta} }\Lambda_{\gamma}\Lambda_{\overline{\delta} }}{\Lambda^{3}}
+\frac{\Lambda_{\alpha \gamma}\Lambda_{\overline{\beta} }\Lambda_{\overline{\delta} }}{\Lambda^{3}}
+\frac{\Lambda_{\overline{\beta} \gamma}\Lambda_{\alpha }\Lambda_{\overline{\delta} }}{\Lambda^{3}}
+\frac{\Lambda_{\alpha \overline{\delta} }\Lambda_{\overline{\beta} }\Lambda_{\gamma}}{\Lambda^{3}}
+\frac{\Lambda_{\overline{\beta} \overline{\delta} }\Lambda_{\alpha }\Lambda_{\gamma}}{\Lambda^{3}}
+\frac{\Lambda_{\gamma\overline{\delta} }\Lambda_{\alpha }\Lambda_{\overline{\beta} }}{\Lambda^{3}}\right)-
\frac{6\Lambda_{\alpha }\Lambda_{\overline{\beta} }\Lambda_{\gamma}\Lambda_{\overline{\delta} }}{\Lambda^{4}}.
\end{align*}
Multiplying both sides by $\psi^{4}$, this equation can be written in a form where $\Lambda$ is multiplied by $\psi^{2n-2}$ and first, second, third and fourth order derivatives of $\Lambda$ are multiplied by $\psi^{2n-1}$, $\psi^{2n}$, $\psi^{2n+1}$ and $\psi^{2n+2}$ respectively. It follows that 
\[
\lim_{z \rightarrow 0} \frac{\partial^{2} g_{\alpha \overline{\beta}}} {\partial z_{\gamma} \partial \overline{z}_{\overline{\delta}}}(z) \big( \psi(z) \big)^{4}\rightarrow 6(2n-2) \psi_{\alpha}(0) \psi_{\overline{\beta}}(0) \psi_{\gamma}(0) \psi_{\overline{\delta}}(0)
\] which is (iii).
\end{proof}

\no We now show that by allowing $z \rightarrow 0$ along the inner normal, it is possible to obtain stronger asymptotics for the functions $g_{\alpha \overline{\beta}}(z)$. Denote by $\mathcal{N}_{0}$ the inner normal to $\partial D$ at $0$ of some fixed length $\epsilon > 0$.

\begin{lem}\label{bdy-psi}
Let $a\in\{1,\ldots,n-1,\overline{1},\ldots,\overline{n-1}\}$. Then
\begin{align*}
\lim_{D\cap \mathcal{N}_{0}\ni z\rightarrow 0}\frac{\psi_{a}(z)}{\psi(z)}=\frac{1}{2}\big(\psi_{a\,n}(0)+\psi_{a\,\overline{n}}(0)\big).
\end{align*}
\end{lem}

\begin{proof}
Expand $\psi$ in a neighbourhood of the origin as
\begin{align}\label{eq1_rho}
\psi(z) = 2 \Re \Big(z_{n} + \frac{1}{2} \sum_{\alpha,\beta=1}^{n} \psi_{\alpha \beta}(0) z_{\alpha}z_{\beta} \Big) + \sum_{\alpha,\beta=1}^{n} \psi_{\alpha \overline{\beta}}(0) z_{\alpha} \overline{z}_{\beta} + o(|z|^{2}).
\end{align}
Differentiating the above equation with respect to $z_{a}$,
\begin{align}\label{eq2_rho}
\psi_{a}(z) = \sum_{\beta=1}^{n} \psi_{a \beta}(0) z_{\beta} + \sum_{\beta=1}^{n} \psi_{a \overline{\beta}}(0) \overline{z}_{\beta} + o(|z|).
\end{align}
Note that for $z \in \mathcal{N}_{0}$ we have $z_{1} = \ldots = z_{n-1} = 0$ and $z_{n} = x_{n}$. Hence from (\ref{eq1_rho}) and (\ref{eq2_rho}) we have
\begin{align*}
\lim_{D \cap \mathcal{N}_{0} \ni z \rightarrow 0} \frac{\psi_{a}(z)}{\psi(z)}
& = \lim_{D \cap \mathcal{N}_{0} \ni z \rightarrow 0} \frac{\psi_{a}(z)/x_{n}}{\psi(z)/x_{n}}\\
& = \lim_{x_{n} \rightarrow 0} \frac{\psi_{a n}(0)+\psi_{a \overline{n}}(0)+o(|x_{n}|)/x_{n}}{2 +
\Re \{\psi_{n n}(0)\} x_{n} + \psi_{n \overline{n}}(0) x_{n} + o(|x_{n}|^{2})/x_{n}}\\
& = \frac{1}{2}\big(\psi_{a n}(0)+\psi_{a\,\overline{n}}(0)\big)
\end{align*}
\end{proof}

\begin{lem}\label{asymptotics-of-Lambda-along-N_0}
Let $a \in \{ 1, \ldots, n-1, \overline{1}, \ldots, \overline{n-1} \}$ and $b \in \{1, \ldots, n, \overline{1}, \ldots, \overline{n} \}$. Then
\begin{enumerate}
\item [(i)] $\displaystyle \lim_{D \cap \mathcal{N}_{0} \ni z \rightarrow 0} \Lambda_{a}(z) \big(\psi(z)\big)^{2n-2}  =  \lambda_{a}(0)+(2n-2)C_{a}$.

\item [(ii)] $\displaystyle \lim_{D \cap \mathcal{N}_{0} \ni z\rightarrow 0} \Lambda_{a b}(z) \big(\psi(z)\big)^{2n-1}  = -(2n-2)\lambda_{a}(0)\psi_{b}(0) -(2n-2)(2n-1)\psi_{b}(0)C_{a} + (2n-2)\psi_{ab}(0)$
\end{enumerate}
where $C_{a} = \frac{1}{2}\big(\psi_{a n}(0)+\psi_{a\,\overline{n}}(0)\big)$.
\end{lem}
\begin{proof}
Consider $\lambda(z) = \Lambda(z) \big( \psi(z) \big)^{2n-2}$ for $z \in D$. The function $\lambda$ is $C^{2}$ on $\overline{D}$ with boundary values $\lambda(z) = - |\nabla \psi(z)|^{2n-2}$ for $z \in \partial D$. In particular $\lambda(0) = -1$. Differentiating $\lambda$ with respect to $z_{a}$ we obtain

\[
\Lambda_{a} \psi^{2n-2} = \lambda_{a} - (2n-2) \frac{1}{\psi} \lambda \psi_{a}
\]
Letting $z \rightarrow 0$ along $\mathcal{N}_{0}$ we obtain
\begin{align}\label{lim-Lambda-i}
\lim_{D \cap \mathcal{N}_{0} \ni z \rightarrow 0} \Lambda_{a}(z) \big( \psi(z) \big)^{2n-2}  =  \lambda_{a}(0) + (2n-2) C_{a}
\end{align}
where
\[ 
C_{a} = \lim_{D\cap \mathcal{N}_{0}\ni z\rightarrow 0}\frac{\psi_{a}(z)}{\psi(z)}=\frac{1}{2}\big(\psi_{a n}(0)+\psi_{a \overline{n}}(0)\big)
\]
by lemma \ref{bdy-psi}. Similarly
\begin{align*}
\Lambda_{a b} \psi^{2n-1} = \, & \lambda_{a b} \psi - (2n-2)
(\lambda_{a}\psi_{b}+\lambda_{b}\psi_{a})\\
& + (2n-2)(2n-1) \frac{1}{\psi} \lambda \psi_{a} \psi_{b}
- (2n-2) \lambda \psi_{ab}
\end{align*}
which gives
\begin{align}\label{lim-Lambda-ij}
\begin{split}
\lim_{D \cap \mathcal{N}_{0} \ni z\rightarrow 0} \Lambda_{a b}(z) \big( \psi(z) \big)^{2n-1}  = & -(2n-2)\lambda_{a}(0)\psi_{b}(0)
-(2n-2)(2n-1)\psi_{b}(0) C_{a} \\ & + (2n-2)\psi_{a b}(0).
\end{split}
\end{align}
\end{proof}


\begin{lem}\label{g_ij*psi}
Let $1\leq \alpha \leq n-1$ and $1\leq \beta \leq n$. Then
\[
\lim_{D \cap \mathcal{N}_{0} \ni z \rightarrow 0} g_{\alpha \overline{\beta}}(z) \big(\psi(z) \big) = (2n-2) \left( \frac{1}{2} \big\{ \psi_{\alpha n}(0) + \psi_{\alpha \overline{n}}(0) \big\} \psi_{\overline{\beta}}(0)-\psi_{\alpha\overline{\beta}}(0)\right).
\]
\end{lem}
\begin{proof}
We have
\[ g_{\alpha \overline{\beta}} = \frac{\partial ^{2} \log(-\Lambda)}{\partial z_{\alpha} \partial \overline{z}_{j}}
= \frac{\Lambda_{\alpha \overline{\beta}}}{\Lambda} - \frac{\Lambda_{\alpha} \Lambda_{\overline{\beta}}}{\Lambda ^{2}}.\]
Multiplying both sides of this equation by $\psi$ we get
\begin{align}\label{g_ij-times-psi}
g_{\alpha \overline{\beta}}\psi^{} & = \frac{\Lambda_{\alpha \overline{\beta}}\psi^{2n-1}}{\Lambda\psi^{2n-2}}-\frac{(\Lambda_{\alpha}\psi^{2n-2})(\Lambda_{\overline{\beta}}\psi^{2n-1})}{(\Lambda\psi^{2n-2})^{2}}.
\end{align}
As in the proof of proposition \ref{lim-g}, 
\[
\Lambda(z) \big( \psi (z) \big)^{2n-2} \rightarrow -1
\]
and 
\[
\Lambda_{\overline{\beta}}(z) \big( \psi(z) \big)^{2n-1} \rightarrow (2n-2) \psi_{\overline{\beta}}(0)
\]
as $z \rightarrow 0$. Therefore letting $z \rightarrow 0$ along $D \cap \mathcal{N}_{0}$ in (\ref{g_ij-times-psi}) and using lemma \ref{asymptotics-of-Lambda-along-N_0} we obtain
\begin{align*}
\lim_{D \cap \mathcal{N}_{0} \ni z \rightarrow 0} g_{\alpha \overline{\beta}}(z) \psi(z)
& = (2n-2) \lambda_{\alpha}(0) \psi_{\overline{\beta}}(0) + (2n-2)(2n-1) \psi_{\overline{\beta}}(0)C - (2n-2) \psi_{\alpha\overline{\beta}}(0)\\
& \hspace{1cm} - \big \{\lambda_{\alpha}(0) + (2n-2) C_{\alpha} \big\} \big \{(2n-2) \psi_{\overline{\beta}}(0) \big\}
\end{align*}
Simplifying the right hand side we obtain
\begin{align*}
\lim_{D \cap \mathcal{N}_{0} \ni z \rightarrow 0} g_{\alpha \overline{\beta}}(z) \big( \psi(z) \big) =
& = (2n-2) \big( \psi_{\overline{\beta}}(0) C_{\alpha} - \psi_{\alpha\overline{\beta}}(0) \big)\\
& = (2n-2) \Big(\frac{1}{2} \big\{\psi_{\alpha n}(0) + \psi_{\alpha \overline{n}}(0) \big\} \psi_{\overline{\beta}}(0) - \psi_{\alpha \overline{\beta}}(0) \Big).
\end{align*}
\end{proof}


\begin{lem}\label{der-g_ij*psi}
Let $1\leq \alpha, \gamma \leq n$ and $1\leq \beta \leq n-1$. Then
\[
\lim_{D\cap \mathcal{N}_{0}\ni z\rightarrow 0}\frac{\partial g_{\alpha \overline{\beta}}}{\partial z_{\gamma}}(z) \big( \psi(z) \big)^{2}
\]
exists and is finite.
\end{lem}
\begin{proof}
Multiplying both sides of (\ref{g-alpha-beta}) by $\psi^{2}$, we obtain
\begin{align*}
\frac {\partial g_{\alpha \overline{\beta}}} {\partial z_{\gamma}} \psi^{2} =  & \frac {\Lambda_{\alpha \overline{\beta} \gamma} \psi^{2n}} {\Lambda \psi^{2n-2}} - \left(\frac {(\Lambda_{\alpha \overline{\beta}} \psi^{2n-1}) (\Lambda_{\gamma} \psi^{2n-1})} {(\Lambda \psi^{2n-2})^{2}} + \frac{(\Lambda_{\alpha \gamma} \psi^{2n}) (\Lambda_{\overline{\beta}}\psi^{2n-2})} {(\Lambda \psi^{2n-2})^{2}} + \frac{(\Lambda_{\overline{\beta} \gamma} \psi^{2n-1}) (\Lambda_{\alpha} \psi^{2n-1})} {(\Lambda \psi^{2n-2})^{2}} \right)\\
& \hspace{2cm}  + \frac{2 (\Lambda_{\alpha} \psi^{2n-1}) (\Lambda_{\overline{\beta}} \psi^{2n-2}) (\Lambda_{\gamma} \psi^{2n-1})} {(\Lambda \psi^{2n-2})^{3}}.
\end{align*}
In view of lemma \ref{asymptotics-of-Lambda-along-N_0} it is seen that the second and third terms have finite limits near the origin along $\mathcal{N}_{0}$ and hence it is enough to study the first term. In particular it suffices to show that
\[
\lim_{D\cap \mathcal{N}_{0}\ni z\rightarrow 0}\Lambda_{\alpha \overline{\beta} \gamma}(z) \big( \psi(z) \big)^{2n}
\]
exists and is finite. To do this we follow the arguments of lemma 6.2 in \cite{LY}. The Green function $G(z,\xi)$ for $D$ with pole at $\xi$ can be written as
\begin{align}\label{G-exp}
G(z, \xi) = |z-\xi|^{-2n+2} + \Lambda(\xi) + H(z, \xi)
\end{align}
where $H(z, \xi)$ is a harmonic function of $z$ for each $\xi \in D$ and satisfies
\begin{align}\label{H-0}
H(\xi,\xi)=0.
\end{align}
The function $G(z,\xi) - |z-\xi|^{-2n+2}$ is real analytic and symmetric in $D \times D$, and harmonic in $z$ and in $\xi$. Since
\[\bigg(\frac{\partial}{\partial z_{\alpha}} + \frac{\partial}{\partial \xi_{\alpha}}\bigg) |z-\xi|^{-2n+2} \equiv 0\]
it follows that the function
\begin{align*}
G_{\alpha}(z,\xi) := \bigg(\frac{\partial G}{\partial z_{\alpha}}+\frac{\partial G}{\partial \xi_{\alpha}}\bigg) (z, \xi)
\end{align*}
is real analytic and symmetric in $D\times D$ and harmonic in both $z$ and $\xi$. Similarly the function
\begin{align*}
G_{\alpha \overline{\beta}}(z,\xi) := \bigg(\frac{\partial G_{\alpha}}{\partial\overline{z}_{\beta}} + \frac{\partial G_{i }} {\partial \overline{\xi}_{\beta}} \bigg)(z, \xi)
\end{align*}
is real analytic and symmetric in $D\times D$ and harmonic in both $z$ and $\xi$. From (\ref{G-exp}) 
\[G_{\alpha}(\xi, \xi) = \frac{\partial \Lambda}{\partial \xi_{\alpha}}(\xi) + \frac{\partial H}{\partial z_{\alpha}}(\xi, \xi) + \frac{\partial H}{\partial \xi_{\alpha}}(\xi, \xi).\]
>From (\ref{H-0}) the function $\tilde{H}(\xi) :=  H(\xi, \xi)$ is identically $0$ on $D$. Therefore
\[\frac{\partial \tilde{H}}{\partial \xi_{\alpha}}(\xi) = \frac{\partial H}{\partial z_{\alpha}}(\xi, \xi) + \frac{\partial H}{\partial \xi_{\alpha}}(\xi, \xi) = 0\]
and hence
\[\frac{\partial \Lambda}{\partial \xi_{\alpha}}(\xi) = G_{\alpha}(\xi,\xi) .\]
Differentiating this with respect to $\xi_{\overline{\beta}}$
\[\frac{\partial ^{2} \Lambda}{\partial \xi_{\alpha} \xi_{\overline{\beta}}} = \frac{\partial G_{\alpha}}{\partial z_{\overline{\beta}}}(\xi, \xi) +  \frac{\partial G_{\alpha}}{\partial \xi_{\overline{\beta}}}(\xi, \xi) = G_{\alpha \overline{\beta}}(\xi, \xi).\]
Differentiating this with respect to $\xi_{\gamma}$
\[\frac{\partial ^{3} \Lambda}{\partial \xi_{\alpha} \overline{\xi}_{\beta} \xi_{\gamma}}(\xi) = \frac{\partial G_{\alpha \overline{\beta}}}{\partial z_{\gamma}}(\xi, \xi) + \frac{\partial G_{\alpha \overline{\beta}}}{\partial \xi_{\gamma}}(\xi, \xi).\]
Now differentiating the symmetry formula $G_{\alpha \overline{\beta}}(z, \xi) = G_{\alpha \overline{\beta}}(\xi, z)$ with respect to $\xi_{\gamma}$ and setting $z = \xi$ we obtain 
\[
\frac{\partial G_{\alpha \overline{\beta}}}{\partial \xi_{\gamma}}(\xi, \xi) = \frac{\partial G_{\alpha \overline{\beta}}}{\partial z_{\gamma}}(\xi, \xi)
\]
which gives
\begin{align}\label{Lambda-G-alpha-beta-gamma}
\Lambda_{\alpha \overline{\beta} \gamma}(\xi) = 2 \frac{\partial G_{\alpha \overline{\beta}}}{\partial z_{\gamma}}(\xi, \xi).
\end{align}
For $\xi \in D$, consider the linear maps $T_{\xi}$ on $\mathbb{C}^{n}$ defined as before by
\begin{align*}
T_{\xi}(z)=\frac{z-\xi}{-\psi(\xi)}
\end{align*}
and let
\begin{align*}
D(\xi)=
\begin{cases}
T_{\xi}(D) & \text{if $\xi \in D$},\\
\left\{w\in\mathbb{C}^{n} : 2 \big(\text{Re}\sum_{ i =1}^{n}\frac{\partial \psi}{\partial z_{ i }}(\xi)w_{ i } \big)-1<0 \right\} & \text{if $\xi \in \partial D$}.
\end{cases}
\end{align*}
Note that $0 \in D(\xi)$ for each $\xi$. We denote by $g(w,\xi)$ the Green's function for $D(\xi)$ with pole at $0$ and define the functions
\begin{align*}
h_{0}(w,\xi) & :=g(w,\xi)+\frac{1}{n-1}\sum_{i=1}^{n}w_{i}\frac{\partial g}{\partial w_{i}}(w,\xi),\\
h_{\alpha}(w,\xi) & :=\psi(\xi) \frac{\partial g}{\partial\xi_{\alpha}}(w,\xi) - (n-1) \psi_{\alpha} (\xi) \big(g_{0}(w,\xi) + \overline{g_{0}(w,\xi)}\big) \text{ and }\\
h_{\alpha \overline{\beta}}(w,\xi) & := -(2n-1) \psi_{\overline{\beta}}g_{\alpha}+ \psi \frac{\partial g_{\alpha}} {\partial \overline{\xi}_{\beta}} - \psi_{\overline{\beta}} \sum_{j = 1}^{n} \Big(w_{j} \frac{\partial g_{\alpha}} {\partial w_{j}}+\overline{w}_{j} \frac{\partial g_{\alpha}}{\partial \overline{w}_{j}}\Big)
\end{align*}
For each $\xi \in \overline{D}$, the above functions are harmonic in $w \in D(\xi)$. Now differentiating 
\[ g(w,\xi) = \big( \psi(\xi) \big)^{2n-2} G(z,\xi)\]
we obtain
\begin{align}\label{g-alpha-G-alpha}
h_{\alpha \overline{\beta}}(w,\xi) = G_{\alpha \overline{\beta}}(z,\xi) \big( \psi(\xi) \big)^{2n} \hspace{.5cm}\text{for $w \in D(\xi),\, \xi \in D$}.
\end{align}
Differentiating this with respect to $z_{\gamma}$ and using the definition of $h_{\alpha \overline{\beta}}$ we obtain
\begin{align*}
\frac{\partial G_{\alpha \overline{\beta}}}{\partial z_{\gamma}}(\xi,\xi)\big(\psi(\xi)\big)^{2n} & =-\frac{1}{\psi(\xi)}\frac{\partial h_{\alpha \overline{\beta}}}{\partial w_{\gamma}}(0,\xi)\\
& = 2n\frac{\psi_{\overline{\beta}}(\xi)}{\psi(\xi)}\frac{\partial h_{\alpha}}{\partial w_{\gamma}}(0,\xi)-\frac{\partial^{2}h_{\alpha}}{\partial w_{\gamma}\partial\overline{\xi}_{\beta}}(0,\xi)
\end{align*}
The results of chapters 4 and 5 of \cite{LY} show that $\frac{\partial h_{\alpha}}{\partial w_{\gamma}}(0,\xi)$ and thus also $\frac{\partial^{2}h_{\alpha}}{\partial w_{\gamma}\partial\overline{\xi}_{\beta}}(0,\xi)$ are continuous at $0\in\partial D$. Hence by lemma \ref{bdy-psi}, we have
\begin{align*}
\lim_{D\cap \mathcal{N}_{0}\ni z\rightarrow 0}\frac{\partial G_{\alpha \overline{\beta}}}{\partial z_{\gamma}}(\xi,\xi)\big(\psi(\xi)\big)^{2n}
\end{align*}
exists and is finite. This together with (\ref{Lambda-G-alpha-beta-gamma}) implies that
\begin{align*}
\lim_{D\cap \mathcal{N}_{0}\ni z\rightarrow 0}\Lambda_{\alpha \overline{\beta} \gamma}(z) \big( \psi(z) \big)^{2n}
\end{align*}
exists and is finite.
\end{proof}

\begin{rem}
The conclusion of lemma \ref{der-g_ij*psi} is not a consequence of theorem 1.1. A direct application of this theorem only yields
\[
 \frac {\partial g_{\alpha \overline{\beta}}} {\partial z_{\gamma}}(z) \backsim \frac {1} {\big( \psi(z) \big)^{3}}
\]
for $z \in D$ close to the origin as can be seen from proposition \ref{lim-g}.
\end{rem}

\begin{lem}\label{det}
The limit
\[
\lim_{D \cap \mathcal{N}_{0} \ni z \rightarrow 0} \det \big( g_{\alpha \overline{\beta}}(z) \big) \big( \psi(z) \big)^{n+1}
\]
exists and is nonzero.
\end{lem}
\begin{proof}
Let $(\Delta_{\alpha \overline{\beta}})$ be the cofactor matrix of $(g_{\alpha \overline{\beta}})$. Then expanding by the $n$-th row,
\[
\det( g_{\alpha \overline{\beta}} ) = g_{n \overline{1}} \Delta_{n \overline{1}} + \ldots + g_{n \overline{n}} \Delta_{n \overline{n}}.
\]
Therefore,
\begin{align}\label{det-times-psi}
\det(g_{\alpha \overline{\beta}}) \psi^{n+1} = (g_{n \overline{1}} \psi^{2}) (\Delta_{n \overline{1}} \psi^{n-1}) + \ldots + (g_{n \overline{n}} \psi^{2}) (\Delta_{n \overline{n}} \psi^{n-1}).
\end{align}
Note that
\begin{align*}
\Delta_{n \overline{\alpha}} \psi^{n-1}
& = \psi^{n-1} (-1)^{n + \alpha} \det
\begin{pmatrix}
g_{1 \overline{1}} & \dots & g_{1 \overline{\alpha-1}} & g_{1 \overline{\alpha + 1}} & \dots & g_{1 \overline{n}}\\
\hdotsfor[1.5]{6}\\
g_{n-1 \overline{1}} & \dots & g_{n-1 \overline{\alpha -1}} & g_{n-1\overline{\alpha+1}} & \dots & g_{n-1\overline{n}}
\end{pmatrix}\\
& = (-1)^{n+\alpha} \det
\begin{pmatrix}
g_{1\overline{1}}\psi & \dots & g_{1\overline{\alpha-1}}\psi & g_{1\overline{\alpha+1}}\psi & \dots & g_{1\overline{n}}\psi\\
\hdotsfor[1.5]{6}\\
g_{n-1 \overline{1}}\psi & \dots & g_{n-1\overline{\alpha-1}}\psi & g_{n-1\overline{\alpha+1}}\psi & \dots & g_{n-1\overline{n}}\psi
\end{pmatrix}
\end{align*}
By lemma \ref{g_ij*psi}, if $1 \leq \alpha \leq n-1$ and $1 \leq \beta \leq n$, then the term $g_{\alpha \overline{\beta}}\psi$ converges to a finite quantity as $z \rightarrow 0$ along $\mathcal{N}_{0}$. It follows that if $1 \leq \alpha \leq n-1$ then
\[
\lim_{D \cap \mathcal{N}_{0} \ni z \rightarrow 0} \Delta_{n \overline{\alpha}}(z) \big( \psi(z) \big)^{n-1}
\]
exists and is finite. Also if $1 \leq \alpha, \beta \leq n-1$, then $g_{\alpha \overline{\beta}}\psi$ converges to $-(2n-2)\psi_{\alpha \overline{\beta}}(0)$. Therefore
\begin{align*}
\lim_{D \cap \mathcal{N}_{0} \ni z \rightarrow 0} \Delta_{n \overline{n}}(z) \big( \psi(z) \big)^{n-1} = (-1)^{n} (2n-2)^{n} \det \big(\psi_{\alpha \overline{\beta}}(0) \big)_{1\leq \alpha, \beta \leq n-1}.
\end{align*}
Finally by proposition \ref{lim-g}, if $1 \leq \alpha, \beta \leq n$, then $g_{\alpha \overline{\beta}}\psi^{2}$ converges to $(2n-2) \psi_{\alpha}(0) \psi_{\overline{\beta}}(0)$. Now it follows from (\ref{det-times-psi}) that
\begin{align*}
\lim_{D \cap \mathcal{N}_{0} \ni z \rightarrow 0} \det \big(g_{\alpha \overline{\beta}}(z) \big) \big( \psi(z) \big)^{n+1} = (-1)^{n}(2n-2)^{n+1} \det \big( \psi_{\alpha \overline{\beta}}(0) \big)_{1\leq \alpha,\beta \leq n-1}\neq 0
\end{align*}
as $D$ is strictly pseudoconvex at $0$.
\end{proof}

\no {\it Proof of Theorem 1.6:} Note that for $z \in \mathcal{N}_{0}$ the first $n-1$ components of $v_{N}(z)$ vanish and hence
\[
R \big( z, v_{N}(z) \big) = \frac{1}{\big( g_{n \overline{n}}(z) \big)^{2}} \bigg(-\frac{\partial^{2} g_{n \overline{n}}}{\partial z_{n} \partial \overline{z}^{n}}(z) + \sum_{\alpha, \beta =1}^{n} g^{\beta \overline{\alpha}}(z) \frac{\partial g_{n\overline{\alpha}}}{\partial z_{n}}(z) \frac{\partial g_{\beta \overline{n}}}{\partial \overline{z}^{n}}(z) \bigg).
\]
We compute the limit of the first term as $z \rightarrow 0$ using proposition \ref{lim-g}. Indeed
\begin{align*}
-\frac {1} {\big( g_{n \overline{n}}(z) \big)^{2}} \frac {\partial^{2} g_{n \overline{n}}} {\partial z_{n} \partial \overline{z}_{n}}(z)
& = -\frac {1} {\Big(g_{n \overline{n}} (z) \big( \psi(z) \big)^{2} \Big)^{2}} \frac{\partial^{2} g_{n \overline{n}}} {\partial z_{n} \partial \overline{z}_{n}}(z) \big( \psi(z) \big)^{4}\\
& \rightarrow -\frac {1} {\big\{ (2n-2) \psi_{n}(0) \psi_{\overline{n}}(0) \big\}^{2} } \big\{ 6(2n-2) \psi_{n}(0) \psi_{\overline{n}}(0) \psi_{n}(0) \psi_{\overline{n}}(0) \big\}\\
& = -\frac{3}{n-1}.
\end{align*}

\medskip

\no To compute the limit of the second term note that $g^{\beta \overline{\alpha}} = \Delta_{\alpha \overline{\beta}} / \det(g_{\alpha \overline{\beta}})$. There are various cases to be considered depending on $\alpha$ and $\beta$.\\

\medskip

\no Case 1:  $\alpha \neq n$, $\beta \neq n$. Here
\begin{align*}
\frac{1}{g_{n \overline{\alpha}}^{2}} g^{\beta \overline{\alpha}} \frac{\partial g_{n \overline{\alpha}}}{\partial z_{n}} \frac{\partial g_{\beta \overline{n}}}{\partial \overline{z}_{n}}
& = \frac{1}{(g_{n\overline{n}}\psi^{2})^{2} \big( \det(g_{i \overline{j}})\psi^{n+1} \big)}(\Delta_{\alpha \overline{\beta}} \psi^{n}) \bigg( \frac{\partial g_{n \overline{\alpha}}}{\partial z_{n}}\psi^{2} \bigg) \left(\frac{\partial g_{\beta \overline{n}}}{\partial\overline{z}_{n}}\psi^{3}\right)
\end{align*}
By lemma \ref{lim-g},
\[
g_{n \overline{n}}(z) \big( \psi^{2}(z) \big) \rightarrow (2n-2)
\]
as $z \rightarrow 0$. By lemma \ref{det}, $\det \big( g_{i \overline{j}}(z) \big) \big( \psi(z) \big)^{n+1}$ converges to a nonzero finite quantity as $z \rightarrow 0$ along $\mathcal{N}_{0}$. Also
\[
\Delta_{\alpha \overline{\beta}} = \sum_{\sigma}(-1)^{sgn(\sigma)}g_{1\overline{\sigma(1)}}g_{2\overline{\sigma(2)}}\cdots g_{n\overline{\sigma(n)}}
\]
where the summation runs over all permutations
\[
\sigma:\{1,\ldots,\alpha-1,\alpha+1,\ldots,n\}\rightarrow \{1,\ldots,\beta-1,\beta+1,\ldots,n\}
\]
Hence
\[
\Delta_{\alpha \overline{\beta}}\psi^{n} = \sum_{\sigma}(-1)^{sgn(\sigma)}(g_{1\overline{\sigma(1)}}\psi)(g_{2\overline{\sigma(2)}}\psi)\cdots (g_{n\overline{\sigma(n)}}\psi^{2}).
\]
By lemma \ref{g_ij*psi}, for $1 \leq i \leq n-1$, $g_{i\overline{\sigma(i)}}(z) \big( \psi(z) \big)$ converges to a finite quantity as $z \rightarrow 0$ along $\mathcal{N}_{0}$. Also
\[
g_{n\overline{\sigma(n)}}(z) \big( \psi(z) \big)^{2} \rightarrow (2n-2) \psi_{n}(0)\psi_{\overline{\sigma(n)}}(0)
\]
as $z \rightarrow 0$ by lemma \ref{lim-g}. Thus $\Delta_{\alpha \overline{\beta}}(z) \big( \psi(z) \big)^{n}$ converges to a finite quantity as $z \rightarrow 0$ along $\mathcal{N}_{0}$.

\medskip

\no By lemma \ref{der-g_ij*psi}, $\frac{\partial g_{n \overline{\alpha}} } {\partial z_{n}}(z) \big( \psi(z) \big)^{2}$ converges to a finite quantity as $z \rightarrow 0$ along $\mathcal{N}_{0}$ and by proposition \ref{lim-g}
\[
\frac{\partial g_{\beta \overline{n}}} {\partial \overline{z}_{n}}(z) \big( \psi(z) \big)^{3} = \overline{\bigg(\frac{\partial g_{n \overline{\beta}}} {\partial z_{n}}(z) \big( \psi(z) \big)^{3} \bigg)}
\rightarrow -2(2n-2) \overline{\big( \psi_{n}(0) \big)} \, \overline{\big( \psi_{\overline{\beta}}(0) \big)} \,  \overline{\big( \psi_{n}(0) \big)} = 0
\]
as $z \rightarrow 0$. Hence
\begin{align*}
\lim_{D\cap \mathcal{N}_{0}\ni z\rightarrow 0} \frac {1} {\big( g_{n\overline{n}}(z) \big)^{2}} g^{\beta \overline{\alpha}}(z) \frac {\partial g_{n\overline{\alpha}}} {\partial z_{n}}(z) \frac {\partial g_{\beta \overline{n}}} {\partial\overline{z}_{n}}(z) = 0.
\end{align*}

\medskip

\no Case 2: $\alpha = n$, $\beta \neq n$. Here
\begin{align*}
\frac{1}{g_{n \overline{n}}^{2}} g^{\beta \overline{n}} \frac{\partial g_{n \overline{n}}}{\partial z_{n}} \frac{\partial g_{\beta \overline{n}}}{\partial \overline{z}_{n}}
& = \frac{1}{(g_{n \overline{n}} \psi^{2})^{2} (\det(g_{i \overline{j}}) \psi^{n+1})} (\Delta_{n \overline{\beta}} \psi^{n-1}) \bigg( \frac{\partial g_{n \overline{n}}}{\partial z_{n}} \psi^{3} \bigg) \bigg( \frac{\partial g_{\beta \overline{n}}}{\partial \overline{z}_{n}} \psi^{3} \bigg)
\end{align*}
By lemma \ref{lim-g}, 
\[
g_{n \overline{n}}(z) \big( \psi(z) \big)^{2} \rightarrow (2n-2)
\]
as $z \rightarrow 0$. By lemma \ref{det}, $\det \big( g_{\alpha \overline{\beta}}(z) \big) \big( \psi(z) \big)^{n+1}$ has a nonzero limit and $\Delta_{n \overline{\beta}}(z) \big( \psi(z) \big)^{n-1}$ converges to a finite quantity as $z \rightarrow 0$ along $\mathcal{N}_{0}$. By proposition \ref{lim-g}, 
\[
\frac{\partial g_{n \overline{n}}} {\partial z_{n}}(z) \big( \psi(z) \big)^{3} \rightarrow -2(2n-2) \psi_{n}(0) \psi_{\overline{n}}(0) \psi_{n}(0) = -2(2n-2)
\]
and
\[
\frac{\partial g_{\beta \overline{n}}} {\partial \overline{z}_{n}}(z) \big( \psi(z) \big)^{3} = \overline{\bigg(\frac{\partial g_{n \overline{\beta}}} {\partial z_{n}}(z) \big( \psi(z) \big)^{3} \bigg)}
\rightarrow -2(2n-2) \overline{\big( \psi_{n}(0) \big)} \, \overline{\big( \psi_{\overline{\beta}}(0) \big)} \,  \overline{\big( \psi_{n}(0) \big)} = 0
\]
as $z \rightarrow 0$. Hence
\[
\lim_{D \cap \mathcal{N}_{0} \ni z \rightarrow 0} \frac {1} {\big( g_{n \overline{n}}(z) \big)^{2}} g^{\beta \overline{n}}(z) \frac {\partial g_{n\overline{n}} } {\partial z_{n}}(z) \frac{\partial g_{\beta \overline{n}}} {\partial\overline{z}_{n}}(z)  = 0
\]

\medskip

\no Case 3: $\alpha \neq n$ and $\beta = n$. This case is similar to Case 2 and we have
\begin{align*}
\lim_{D\cap \mathcal{N}_{0}\ni z\rightarrow 0}\frac{1}{\big(g_{n \overline{n}}(z) \big)^{2}} g^{n\overline{\alpha}}(z) \frac{\partial g_{n\overline{\alpha}}}{\partial z_{n}}(z) \frac{\partial g_{n\overline{n}}}{\partial\overline{z}_{n}}(z)
& = 0.
\end{align*}

\medskip

\no Case 4: $\alpha = n$, $\beta = n$. In this case we have
\begin{align*}
\frac{1}{g_{n\overline{n}}^{2}}g^{n\overline{n}}\frac{\partial g_{n\overline{n}}}{\partial z_{n}}\frac{\partial g_{n\overline{n}}}{\partial\overline{z}_{n}}
& = \frac{1}{(g_{n\overline{n}}\psi^{2})^{2} \big( \det(g_{i \overline{j}})\psi^{n+1} \big)}(\Delta_{n\overline{n}}\psi^{n-1})
\bigg( \frac{\partial g_{n\overline{n}}}{\partial z_{n}}\psi^{3} \bigg) \bigg( \frac{\partial g_{n\overline{n}}}{\partial\overline{z}_{n}}\psi^{3} \bigg).
\end{align*}
>From lemma \ref{lim-g}, 
\[
g_{n\overline{n}}(z) \big(\psi(z) \big)^{2} \rightarrow (2n-2)
\]
and both
\[
\frac {\partial g_{n \overline{n}} } {\partial z_{n}}(z) \big( \psi(z) \big)^{3}, \frac{\partial g_{n \overline{n}} } {\partial \overline{z}_{n}}(z) \big( \psi(z) \big)^{3} \rightarrow -2(2n-2)
\]
as $z \rightarrow 0$. From lemma \ref{det}
\[\Delta_{n \overline{n}}(z) \big( \psi(z) \big)^{n-1} \rightarrow (-1)^{n} (2n-2)^{n} \det \big(\psi_{i \overline{j}}(0) \big)_{1 \leq i, j \leq n-1}
\]
and 
\[
\det \big( g_{i \overline{j}}(z) \big) \big( \psi(z) \big)^{n+1} \rightarrow  (-1)^{n} (2n-2)^{n+1} \det \big( \psi_{i \overline{j}}(0) \big)_{1 \leq i, j \leq n-1}
\]
as $z \rightarrow 0$ along $\mathcal{N}_{0}$. Hence
\begin{align*}
\lim_{D \cap \mathcal{N}_{0} \ni z \rightarrow 0} \frac {1} {\big(g_{n \overline{n}}(z) \big)^{2}} g^{n \overline{n}}(z) \frac{\partial g_{n \overline{n}}} {\partial z_{n}}(z) \frac{\partial g_{n \overline{n}}} {\partial\overline{z}_{n}}(z)
& = \frac{2}{n-1}.
\end{align*}
>From the various cases we finally obtain
\begin{align*}
\lim_{D\cap \mathcal{N}_{0}\ni z\rightarrow 0} R\big( z,v_{N}(z) \big)=-3/(n-1) + 2/(n-1) = -1/(n-1).
\end{align*}

\begin{rem}
To understand the difficulty in computing the holomorphic sectional curvature along tangential directions let us consider the following example: let $z \in \mathcal{N}_{0}$ and $v = (1, \ldots, 0)$. Then $v_{T}(z) \equiv (1, \ldots, 0)$ and hence
\[
R\big( z, v_{T}(z) \big) = \frac {1} {\big( g_{1 \overline{1}}(z) \big)^{2} } \bigg(-\frac{\partial^{2} g_{1 \overline{1}}} {\partial z_{1} \partial \overline{z}_{1}}(z) + \sum_{\alpha, \beta =1}^{n} g^{\beta \overline{\alpha}}(z) \frac{\partial g_{1 \overline{\alpha}} }{\partial z_{1}} \frac{\partial g_{\beta \overline{1}}} {\partial \overline{z}_{1}}(z) \bigg).
\]
By lemma \ref{g_ij*psi},
\[
\lim_{D \cap \mathcal{N}_{0} \ni z \rightarrow 0}g_{1 \overline{1}}(z)  \psi(z)  = -(2n-2)\psi_{1 \overline{1}}(0) \neq 0
\]
as $D$ is strictly pseudoconvex at the origin and hence
\[
g_{1 \overline{1}}(z) \backsim \frac {1} { \psi(z) }
\]
for $z \in D \cap \mathcal{N}_{0}$ near the origin. Therefore we require that
\[
\frac{\partial^{2} g_{1 \overline{1}}} {\partial z_{1} \partial \overline{z}_{1}}(z) \backsim \frac{1} {\big( \psi(z) \big)^{2}}
\]
for $z \in D \cap \mathcal{N}_{0}$ near the origin. This asymptotic is much sharper than the one obtained from theorem 1.1 which is
\[
\frac{\partial^{2} g_{1 \overline{1}}} {\partial z_{1} \partial \overline{z}_{1}}(z) \backsim \frac{1} {\big( \psi(z) \big)^{4}}
\]
for $z \in D$ near the origin as computed in lemma \ref{lim-g}. Thus to compute the holomorphic sectional curvature along tangential directions a stronger claim about the blow up of the fourth order derivative of $\Lambda$ near the boundary is needed.
\end{rem}


\section{Convex domains with constant negative holomorphic sectional curvature}
\no In this section $D$ will be a smoothly bounded strongly convex domain in $\mathbb{C}^{n}$ equipped with the $\Lambda$-metric. Under the assumption that $ds^{2}_{z}$, i.e., the $\Lambda$-metric has constant negative holomorphic sectional curvature the goal wll be to prove that $D \simeq \mathbb{B}^{n}$ and that the $\Lambda$-metric is proportional to the pull back of the Bergman metric in $\mathbb{B}^{n}$. Note that
\[ \lim_{z \ra \partial D} F_{D}^{R}(z, v) = + \infty\]
for $v \in \mathbb{C}^{n}$ with $|v| = 1$ which is a consequence of theorem 1.3. Let $ds^{2}_{\Delta}$ be the Bergman metric on $\Delta$.
\begin{lem}
Let $D$ be as above and suppose that the holomorphic sectional curvature of the $\Lambda$-metric is bounded from below by a negative constant $-c^{2}$. Then for every holomorphic map $f : D \ra \Delta$, we have
\[ \Big( f^{*} \big( ds^{2}_{\Delta} \big) \Big)_{z} \leq \frac{c^{2}}{4} \, ds^{2}_{z} \]
for $z \in D$.
\end{lem}

\begin{proof}
Without loss of generality let us assume that $D$ contains the origin. Let $\alpha\in (0,1)$. Consider the shrinking map $g_{\alpha} : D \ra \mathbb{C}^{n}$ defined by $g_{\alpha}(z)=\alpha z$ and let $D_{\alpha} = g_{\alpha}(D)$. The mapping $g_{\alpha}$ is a biholomorphism of $D$ onto $D_{\alpha}$. Hence we define the Hermitian metric $ds^{2}_{\alpha}$ on $D_{\alpha}$ by setting
\[
\big( ds^{2}_{\alpha} \big)_{z} := \big( g_{\alpha}^{-1} \big)^{*} \big( ds^{2}_{z} \big).
\]
Observe that $ds^{2}_{\alpha}$ blows up near the boundary of $D_{\alpha}$. Indeed, let $z\in D_{\alpha}$ and $v\in\mathbb{C}^{n}$ with $|v|=1$. Then from the definition of $ds^{2}_{\alpha}$, we have
\[
F^{R}_{D_{\alpha}}(z,v) = \frac{1}{\alpha} \, F^{R}_{D} (z / \alpha, v).
\]
Since $z \ra \partial D_{\alpha}$ implies that $z / \alpha = g_{\alpha}^{-1}(z)\ra \partial D$, we obtain from the above equation that
\[
\lim_{z \ra \partial D_{\alpha}} F^{R}_{D_{\alpha}}(z,v)
= \lim_{z / \alpha \ra \partial D} \frac{1}{\alpha} \, F^{R}_{D} (z / \alpha, v) = + \infty.
\]
Note that $D_{\alpha}\subset\subset D$ as $D$ is convex. We now show that the metric $ds^{2}_{\alpha}$ actually converges to $ds^{2}_{z}$  as $\alpha \ra 1$ at each point in $D$. Let $z\in D$ and take $\alpha$ sufficiently close to $1$ so that $z\in D_{\alpha}$. Then for any $v, w \in \mathbb{C}^{n}$ it is evident that
\[
\big( ds^{2}_{\alpha} \big) _{z}(v,w) = \frac{1}{\alpha^{2}} \, ds^{2}_{z / \alpha}(v,w).
\]
Since $ds^{2}_{z}$ is a Hermitian metric, for fixed $v,w\in\mathbb{C}^{n}$, $ds^{2}_{p}(v,w)$ defines a smooth function of $p \in D$. Hence from the above equation it follows that $\big( ds^{2}_{\alpha} \big) _{z} \ra ds^{2}_{z}$ as $\alpha \ra 1$.

\medskip

The reason for constructing the metric $ds^{2}_{\alpha}$ is clear now. The lemma would be proved if we show that 
\[
f^{*} \big( ds^{2}_{\Delta} \big) \leq \frac{c^{2}}{4} \, ds^{2}_{\alpha}
\]
for each $\alpha \in (0,1)$. So let us fix $\alpha\in(0,1)$. Let $S_{n}=\{v\in\mathbb{C}^{n}:|v|=1\}$. Consider the function
\[
 u(z,v):= \frac{\big( f^{*}(ds^{2}_{\Delta}) \big)_{z}(v,v)}{\big( ds^{2}_{\alpha} \big)_{z}(v,v)}
\]
defined on $D_{\alpha} \times S_{n}$. The numerator is bounded on $\overline{D}_{\alpha} \times S_{n}$ and the denominator blows up on $\partial D_{\alpha} \times S_{n}$. Therefore the nonnegative function $u$ goes to zero at $\partial D_{\alpha}\times S_{n}$. In particular, $u$ attains its maximum at $(z_{0},v_{0})\in D_{\alpha}\times S_{n}$. If $u(z_{0},v_{0})=0$ then $u \equiv 0$ and in this case
\[
f^{*} \big( ds^{2}_{\Delta} \big) \leq \frac{c^{2}}{4} \,ds^{2}_{\alpha}.
\]
is evidently true. Therefore, we assume that $u(z_{0},v_{0})>0$. Let $S$ be a complex submanifold of $D_{\alpha}$ of dimension $1$ through $z_{0}$ and tangent to $v_{0}$, such that the holomorphic sectional curvature $H_{ds^{2}_{\alpha}}(z_{0},v_{0})$ of $ds^{2}_{\alpha}$ at $z_{0}$ in the direction of $v_{0}$ is equal to the Gaussian curvature $H_{\left. ds^{2}_{\alpha} \right| _{S}}(z_{0})$ of the restriction of $ds^{2}_{\alpha}$ on $S$ at $z_{0}$.

\medskip

Let $\xi$ be a local coordinate system around $z_{0}$ on $S$. Let $f^{*} \big( ds^{2}_{\Delta} \big) \big| _{S} = 2g \, d\xi \, d\overline{\xi}$ and  $\big( ds^{2}_{\alpha} \big) \big| _{S} = 2h \, d\xi \, d\overline{\xi}$. Then the restriction of $u$ on the tangent bundle $T(S)$ of $S$ is equal to $g / h$ on this coordinate system. The Gaussian curvature of the metric $f^{*} \big( ds^{2}_{\Delta} \big) \big| _{S} = 2 g \, d\xi \, d\overline{\xi}$ is given by
\[
H_{f^{*} \big( ds^{2}_{\Delta} \big) \big| _{S}} = -\frac{1}{g} \, \frac{\partial^{2}\log g}{\partial \xi \partial \overline{\xi}}.
\]
Since $u(z_{0}, v_{0}) > 0$, the mapping $f$ is a submersion near $z_{0}$. Therefore the holomorphic sectional curvature of $f^{*} \big( ds^{2}_{\Delta} \big)$ at $z_{0}$ along any vector $v \in \big(\ker df(z_{0}) \big) ^ {\perp}$ is equal to $-4$. In particular, since $v_{0} \in \big(\ker df(z_{0}) \big) ^ {\perp}$, we have
\[
H_{f^{*} \big( ds^{2}_{\Delta} \big) \big| _{S}} (z_{0}) \leq H_{f^{*} \big( ds^{2}_{\Delta} \big)} (z_{0}, v_{0}) = -4.
\]
Thus
\begin{equation}
\frac{\partial^{2}\log g}{\partial \xi \partial \overline{\xi}}(z_{0}) \geq 4 g(z_{0}).
\end{equation}
\no The Gaussian curvature of the metric $\big( ds^{2}_{\alpha} \big) \big| _{S} = 2h \, d\xi \, d\overline{\xi}$ is given by
\[
H_{\big( ds^{2}_{\alpha} \big) \big| _{S}} = -\frac{1}{h} \, \frac{\partial^{2} \log h}{\partial \xi \partial \overline{\xi}}.
\]
By our assumption
\[
H_{\big( ds^{2}_{\alpha} \big) \big| _{S}} (z_{0}) = H_{ds^{2}_{\alpha}} (z_{0},v_{0}) \geq -c^{2}
\]
and hence
\begin{equation}
\frac{\partial^{2} \log h}{\partial \xi \partial \overline{\xi}}(z_{0}) \leq c^{2} h(z_{0}).
\end{equation}
Since $u$ attains a maximum at $(z_{0},v_{0})$, we obtain from (6.1) and (6.2) that
\[
0 \geq \frac{\partial^{2} \log u}{\partial \xi \partial \overline{\xi}}(z_{0}) = \frac{\partial^{2} \log g}{\partial \xi \partial \overline{\xi}}(z_{0}) - \frac{\partial^{2} \log h}{\partial \xi \partial \overline{\xi}}(z_{0})
\geq 4 g(z_{0})-c^{2} h(z_{0}).
\]
This implies that 
\[
u(z_{0},v_{0}) = g(z_{0})/h(z_{0}) \leq c^{2}/4.
\]
Thus $u \leq c^{2}/4$ and hence
\[
f^{*} \big( ds^{2}_{\Delta} \big) \leq \frac{c^{2}}{4} \, ds^{2}_{\alpha}
\]
which completes the proof.
\end{proof}

\no {\it Proof of Theorem 1.6:} Suppose the holomorphic sectional curvature of $ds^{2}_{z}$ is equal to $-c^{2}$. We divide the proof of the theorem into two parts. In the first part using only the fact that the holomorphic sectional curvature of $ds^{2}_{z}$ is bounded above by $-c^{2}$ we show that the inequality 
\begin{equation}
F^{K}_{D}(z, v) \geq \frac{c}{2} \, F^{R}_{D}(z, v)
\end{equation}
holds for $z \in D$ and $v$ a tangent vector at $z$. The convexity and blow up of the length of a vector at the boundary does not play any role here. The second part is based on lemma 6.1 and using $-c^{2}$ as a lower bound for the curvature we show that
\begin{equation}
\frac{c}{2} \, F^{R}_{D}(z, v) \geq F^{C}_{D}(z, v)
\end{equation}
for $z \in D$ and $v$ a tangent vector at $z$.

\medskip

\no Assuming that we have shown the above two inequalities, the proof of the theorem can be completed as follows. Since $D$ is bounded and convex, Lempert's theorem shows that the Carath\'{e}odory and the Kobayashi metrics are equal and hence by (6.3) and (6.4) they are both smooth hermitian metrics. Also boundedness and convexity of $D$ implies that it is complete hyperbolic with respect to the Kobayashi metric. Thus it follows from the main theorem of \cite{S} that $D$ is biholomorphic to the unit ball in $\mathbb{C}^{n}$. It remains to show that $ds^{2}_{z}$ is proportional to the Bergman metric on $D$. Let $f : D \ra \mathbb{B}^{n}$ be a biholomorphism. Then for any $z \in D$ and $v \in \mathbb{C}^{n}$,
\[
\frac{c}{2}F^{R}_{D}(z,v) = F^{K}_{D}(z,v) = F^{K}_{\mathbb{B}^{n}} \big( f(z), df(z)v \big) = F^{B}_{\mathbb{B}^{n}} \big( f(z), df(z)v \big) = F^{B}_{D}(z,v).
\]
where the superscript $B$ on $F_{D}$ and $F_{\mathbb{B}^{n}}$ denotes the Bergman metric on $D$ and $\mathbb{B}^{n}$ respectively. This proves the claim.

\medskip

We now establish the two inequalities in (6.3) and (6.4). To prove (6.3), let $f: \Delta \ra D$ be a holomorphic map. Since the Gaussian curvature of the Poincar\'{e} metric is $-4$, the Ahlfors-Schwarz lemma shows that
\[
f^{*} \big( ds^{2}_{z} \big) \leq \frac{4}{c^{2}} \, ds^{2}_{\Delta}.
\]
This implies that for any $z \in \Delta$ and $v \in \mathbb{C}$,
\[
\frac{c}{2} \, F^{R}_{D} \big( f(z), f^{\prime}(z)v \big) \leq |v|_{ds^{2}_{\Delta}}.
\]
Since the Kobayashi metric is the largest of all metrics that are distance decreasing for holomorphic maps $f: \Delta \ra D$, it follows that
\[
F^{K}_{D} \geq \frac{c}{2} \, F^{R}_{D}
\]
which is (6.3).
\medskip

To prove (6.4), let $f:D \ra \Delta$ be a holomorphic mapping.  By lemma 6.1 we have
\[
f^{*} \big( ds^{2}_{\Delta} \big) \leq \frac{c^{2}}{4} \, ds^{2}_{z}.
\]
This implies that for any $z \in D$ and $v \in \mathbb{C}^{n}$
\[
\big| df(z) v \big| _{ds^{2}_{\Delta}} \leq \frac{c}{2} \, F^{R}_{D}(z,v).
\]
Since the Carath\'{e}odory metric is the smallest of all metrics that are distance decreasing for holomorphic maps $f : D \ra \Delta$, it follows that
\[
\frac{c}{2} \, F^{R}_{D} \geq F^{C}_{D}
\]
which is (6.4).

\medskip

This completes the proof of the theorem.

\section{Stability of the $\Lambda$-metric under $C^{2}$ perturbation and Concluding Remarks}

\no In this section we study the interior stability of the $\Lambda$-metric under perturbations of a given domain. Let $D$ be a domain in $\mathbb{C}^{n}$ with $C^{2}$-smooth boundary. By a $C^{2}$ perturbation of $D$ we mean a sequence $\{D_{j}\}$ of domains in $\mathbb{C}^{n}$ which converges in the $C^{2}$ topology to $D$. Let $G$ be the Green function for $D$ and $\Lambda$ the associated Robin function. Likewise let $G_{j}$ be the Green function for $D_{j}$ and $\Lambda_{j}$ the corresponding Robin function. 
\begin{prop}
For every $p \in D$, $G_{j}(z, p)$ converges uniformly on compact subsets of $D$ to $G(z, p)$.
\end{prop}
\begin{proof}
Let $p \in D$. Then $p$ is contained in $D_{j}$ for large $j$ so that the function $G_{j}(z,p)$ is well defined on $D_{j}$ for all such $j$. The function $G_{j}(z,p)$ is harmonic on $D_{j}\setminus\{p\}$ and satisfies
\begin{align}\label{G_j-bounded1}
0 \leq G_{j}(z,p) \leq \frac{1}{|z-p|^{2n-2}}
\end{align}
there by the maximum principle. If $K \subset D$ is compact then both $p$ and $K$ are contained in $D_{j}$ for large $j$ and by (\ref{G_j-bounded1}), $\{G_{j}(z,p)\}$ is uniformly bounded on $K$. Hence there exists a subsequence which converges uniformly on compact subsets of $D \setminus \{p\}$.

\medskip

Thus to prove that $G_{j}(z, p)$ converges uniformly on compacts of $D \setminus \{p\}$ to $G(z, p)$ it suffices to show that any convergent subsequence of $\{G_{j}(z, p)\}$ has the unique limit $G(z, p)$. Let $\{G_{j_{\nu}}(z,p)\}_{j_{\nu} \in \mathbb{N}}$ be a subsequence which converges uniformly on compact subsets of $D \setminus \{p\}$ to a function, say $\tilde{G}(z, p)$. Then $\tilde{G}(z, p)$ is a harmonic function on $D \setminus \{p\}$ and it follows from the proof of proposition 3.3 that the function  $\tilde{G}(z, p) - |z - p|^{-2n+2}$ is harmonic near $p$. To show that $\tilde{G}(z, p) \rightarrow 0$ as $z \rightarrow \partial D$, note that
\[ 0 \leq \tilde{G}(z,p) \leq \frac{1}{|z-p|^{2n-2}} \]
for all $z \in D \setminus \{p\}$ since the same holds for $G_{j}(z,p)$. For a given $\epsilon > 0$ it is therefore possible to choose a large ball $B(p, R)$ such that $\tilde{G}(z, p) < \epsilon$ for $z \in D \setminus B(p, R)$. On the other hand as in proposition 3.5 it is possible to show that there is a finite cover of $\partial D \cap \overline{B(p, R)}$ by open balls in which the estimate
\[
\tilde{G}(z, p) \lesssim \delta(z) \, \vert z - p \vert^{-2n + 1}
\]
holds. Combining these observations it follows that if $z$ is close enough to $\partial D$ then $\tilde{G}(z, p) < \epsilon$ and this proves that $\tilde{G}(z,p)$ is the Green function for $D$ with pole at $p$.
\end{proof}

\medskip

\no For multi-indices $A = (\alpha_{1}, \ldots, \alpha_{n})$ and $B = (\beta_{1}, \ldots, \beta_{n}) \in \mathbb{N}^{n}$, $D^{A}$ and $D^{B}$ will have the same meaning as in theorem $1.1$ and $D^{A \overline{B}} = D^{A} D^{\overline{B}}$.

\begin{prop}
For multi-indices $A, B \in \mathbb{N}^{n}$
\[
D^{A \overline{B}} \Lambda_{j} \rightarrow D^{A \overline{B}} \Lambda
\]
uniformly on compact subsets of $D$.
\end{prop}
\begin{proof}
Fix a compactly contained ball $U = B(0, r) \subset D$ and note that $U \subset D_{j}$ for all large $j$. For each such $j$ and $p \in U$, lemma 2.2 shows that
\[
\Lambda_{j}(p) = \frac{1}{(\sigma_{2n} r)^{2}} \iint_{\partial U \times \partial U} H_{j}(z, w) \frac{\big(r^{2} - |z-p|^{2} \big) \big( r^{2} - |w-p|^{2} \big) }{|z-p|^{2n}|w-p|^{2n}} \, dS_{z}\, dS_{w}.
\]
where $H_{j}(z, w) = G_{j}(z, w) - |z - w|^{-2n+2}$. Differentiating with respect to $p$ under the integral sign we obtain
\[
D^{A \overline{B}}\Lambda_{j}(p) = \frac{1}{(\sigma_{2n} r)^{2}} \iint_{\partial U \times \partial U} H_{j}(z, w) \, D^{A \overline{B}} \, \frac{\big(r^{2} - |z-p|^{2} \big) \big( r^{2} - |w-p|^{2} \big) }{|z-p|^{2n}|w-p|^{2n}} \, dS_{z}\, dS_{w}.
\]
Now as in proposition 3.6, $H_{j}(z, w)$ converges uniformly on compact subsets of $D \times D$ to the function $H_{\infty}(z, w) = G(z, w) - |z - w|^{-2n+2}$. Therefore the integral above converges to
\[
\frac{1}{(\sigma_{2n} r)^{2}} \iint_{\partial U \times \partial U} H_{\infty}(z, w) \, D^{A \overline{B}} \, \frac{\big( r^{2} - |z-p|^{2} \big) \big( r^{2} - |w-p|^{2} \big) }{|z-p|^{2n}|w-p|^{2n}} \, dS_{z}\, dS_{w} = D^{A \overline{B}} \Lambda(p).
\]
Moreover the limit is uniform in $p \in U$ which completes the proof.
\end{proof}

\no {\it Proof of Theorem 1.8:} From proposition 7.2 the components of the $\Lambda$-metric for $D_{j}$ and their derivatives converge uniformly on compact subsets of $D$ to the corresponding components of the  $\Lambda$-metric for $D$ and their derivatives. This proves the theorem.

\medskip

We conclude this article with a few open problems involving the $\Lambda$-metric. First of all given a bounded strongly pseudoconvex domain $D$ with smooth boundary, we have shown in theorem 1.4 that the $\Lambda$-metric is comparable to the Kobayashi metric and hence is complete on $D$. The completeness of the  $\Lambda$-metric on such a domain $D$ was first proved by Yamaguchi and Levenberg in \cite{LY} by showing that for every curve $\gamma : [0, 1) \ra D$ that approaches $\partial D$ as $t \ra 1^{-}$ one has
\begin{equation}
\int_{\gamma} \, ds^{2}_{z} = + \infty.
\end{equation}
But for an arbitrary smoothly bounded pseudoconvex domain they were able to prove (7.2) only for those curves $\gamma$ which approach the boundary non-tangentially (cf. theorem 6.1 in \cite{LY}) or which have finite Euclidean length (cf. corollary 9.1 in \cite{LY}). Thus it is unclear whether the $\Lambda$-metric is complete for an arbitrary smoothly bounded pseudoconvex domain in $\mathbb{C}^{n}$.

\medskip

It is unknown if the $\Lambda$-metric is invariant under biholomorphisms. That it is so under affine maps of the form $z \mapsto a U z + b$ where $a \in \mathbb{C}$, $a \neq 0$, $b \in \mathbb{C}^{n}$ and $U$ is a complex unitary matrix has been checked in lemma 5.1.

\medskip

Let $D$ be a smoothly bounded strongly pseudoconvex domain in $\mathbb{C}^{n}$ and let $\{z_{j}\}$ be a sequence of points in $D$ which converges to a boundary point $z_{0} \in \partial D$. Further assume that the points $z_{j}$ lie on the inner normal to $\partial D$ at $z_{0}$. Then we have shown in theorem 1.6 that the holomorphic sectional curvature of the $\Lambda$-metric at $z_{j}$ along normal directions converges to the constant $-1/(n-1)$. But the behaviour of the holomorphic sectional curvature along normal directions is unclear if we remove the above restriction on $z_{j}$. Also the difficulty in computing the holomorphic sectional curvature along tangential directions has been discussed in remark 5.9.

\medskip

Let $D$ be a smoothly bounded strongly pseudoconvex domain in $\mathbb{C}^{n}$. Let $\tilde{\Delta}$ be the Laplace-Beltrami operator associated to the $\Lambda$-metric. It will be interesting to compute the boundary asymptotics for $\tilde{\Delta}$.

\medskip

Let $D$ be a smoothly bounded pseudoconvex domain in $\mathbb{C}^{n}$. Let $K_{D}(z)$ be the Bergman kernel associated to $D$ and let $b_{\alpha \ov \beta}(z)$ be the components of the Bergman metric, i.e.,
\[
b_{\alpha \ov \beta}(z)  =  \frac{\partial^{2} \log K_{D}}{\partial z_{\alpha} \partial \ov z_{\beta}}(z).
\]
Let $B_{D}(z) = \big( b_{\alpha \ov \beta}(z) \big)$. Then the Bergman canonical invariant is defined by
\[
J_{D}^{B}(z) = \frac{\det B_{D}(z)}{K_{D}(z)}
\]
and the Ricci curvature is defined by
\[
R_{D}^{B} = \sum_{\alpha, \beta = 1}^{n} R_{\alpha \ov \beta} \, dz_{\alpha} \, d \ov{z}_{\beta} \phantom{mm} 
\text{where} \phantom{mm}
R_{\alpha \ov \beta} = - \frac{\partial^{2} \log \det G_{D}}{\partial z_{\alpha} \partial \ov{z}_{\beta}}.
\]
The boundary behavior of these two objects on $h$-extensible domains were studied in \cite{KY}. Similar objects $J_{D}^{R}$ and $R_{D}^{R}$ can be defined 
for the $\Lambda$-metric. Note that $J_{D}^{B}$ and $R_{D}^{B}$ are invariant under 
biholomorphisms of $D$ whereas the behavior of $J_{D}^{R}$ and $R_{D}^{R}$ under such maps is not clear. Nevertheless, it will be interesting
to find the boundary behavior of $J_{D}^{R}$ and $R_{D}^{R}$.

\medskip

Finally, let $D$ be a smoothly bounded strongly pseudoconvex domain in $\mathbb{C}^{n}$. For $z \in D$ and a unit vector $v \in \mathbb{C}^{n}$ denote by 
$X(t, z, v)$ the geodesic in the $\Lambda$-metric starting at $t = 0$ at the point $z$ with initial velocity $v$. 
A general question is the following: given $z$ close to $\partial D$ and unit vector $v$, does the geodesic $X(t, z, v)$ hit $\partial D$ 
at a unique point provided that $X(t, z, v)$ does not remain in a fixed compact subset of $D$ for all $t \geq 0$? For the Bergman metric this was done by Fefferman 
in \cite{F}.

\end{document}